\documentclass[11pt, reqno]{amsart}
\usepackage{amsmath}
\usepackage{amsthm}
\usepackage{amssymb}
\usepackage{amsmath}
\usepackage{amsthm}
\usepackage{amssymb}
\usepackage{fancyhdr}
\usepackage{enumerate}
\usepackage{amscd}
\usepackage[all]{xy}
\usepackage{graphicx}
\usepackage{cite}

\theoremstyle{remark}
\newtheorem*{Remark}{Remark}

\newtheorem*{pfe}{Proof of Theorem \ref{EXIST}}
\newtheorem*{pflu}{Proof of Theorem \ref{LU}}
\newtheorem*{cpf}{Proof of the claim}
\newtheorem*{pfl}{Proof of the lemma}

\theoremstyle{plain}
\newtheorem{THM}{Theorem}

\newtheorem{lemma}{Lemma}[section]

\newtheorem{thm}[lemma]{Theorem}

\newtheorem{defi}[lemma]{Definition}
\newtheorem*{rmk}{Remark}
\newtheorem{cor}[lemma]{Corollary}
\numberwithin{equation}{section}

\newtheorem*{SobIneq}{Sobolev Inequality}

\begin{document}
\title{Foliations by Stable Spheres with Constant Mean Curvature for Isolated Systems with General Asymptotics}
\author{Lan--Hsuan Huang}
\address{Department of Mathematics\\Stanford University\\ Stanford, CA~94305}
\email{lhhuang@math.stanford.edu}
\thanks{The global uniqueness result was completed at Institut Mittag-Leffler in Autumn 2008 when the author participated the program {\it Geometry, Analysis, and General Relativity}. The author is grateful for their hospitality and the generous support.}
\thanks{{\it Current address}: Department of Mathematics, Columbia University, 2990 Broadway, New York, NY 10027, USA. E-mail: lhhuang@math.columbia.edu}
\date{April, 2010}

\begin{abstract}
We prove the existence and uniqueness of constant mean curvature foliations for initial data sets which are asymptotically flat satisfying the Regge--Teitelboim condition near infinity. It is known that the (Hamiltonian) center of mass is well-defined for manifolds satisfying this condition. We also show that the foliation is asymptotically concentric, and its geometric center is the center of mass. The construction of the foliation generalizes the results of Huisken--Yau, Ye, and Metzger, where strongly asymptotically flat manifolds and their small perturbations were studied.
\end{abstract}
\maketitle
\pagestyle{myheadings}
\markright{Foliations by Stable Spheres with Constant Mean Curvature}

\section{Introduction}

Whether a foliation of constant mean curvature surfaces uniquely exists in an exterior region of an asymptotically flat manifold is a fundamental problem in general relativity. The significance of this problem is that the foliation provides an intrinsic geometric structure near infinity, supplies a definition of the center of mass in general relativity, and has a relation to quasi-local mass.

Currently, a widely-used definition of asymptotic flat manifolds at infinity is expressed in terms of Cartesian coordinates outside a compact set and requires suitable decay rates on the data. The definition is convenient for calculation purposes, but it may obscure interesting geometry \cite[p.697]{Yau82}. In order to understand the canonical structure of asymptotically flat manifolds, Yau suggests that the constant mean curvature foliation is a promising description of asymptotic flat manifolds near infinity. Moreover, once the foliation exists and is unique, one can develop polar coordinates analogous to the polar coordinates in Euclidean space, and a canonical concept of center of mass can been defined. Also, the Hawking mass is a quantity introduced to capture the energy content of the region bounded by a two-surface $N$ which is defined as follows:
\begin{align*}	
		m_H (N) = \frac{|N|^{\frac{1}{2} }}{ ( 16 \pi )^{\frac{3}{2}} } \left( 16\pi - \int_N H^2 \, d\sigma \right).
\end{align*}
Christodoulou and Yau \cite{CY86} proved that the Hawking mass is non-negative on a stable surface with constant mean curvature for initial data sets satisfying the dominant energy condition. Bray \cite{Bray01} showed that the Hawking mass is monotonically increasing along the isoperimetric constant mean curvature surfaces and converges to the ADM mass at infinity.

For the existence and uniqueness of constant mean curvature foliation, some results have been achieved for \emph{strongly} asymptotically flat manifolds whose metrics, in some asymptotically flat coordinate chart, are of the form:
\begin{eqnarray} \label{SAF}
		&&g_{ij}(x) = \left( 1 + \frac{ 2m }{ | x | }\right) \delta_{ij} + p_{ij}, \notag \\
		&&\qquad \qquad p_{ij}(x) = O( |x|^{-2}), \partial^{\alpha} p_{ij}(x) = O( |x|^{ -2 - |\alpha| }), 
\end{eqnarray}
where $m$ is the ADM mass. 

Huisken and Yau \cite{HY96} proved the existence of constant mean curvature foliations for strongly asymptotically flat manifolds, if $m>0$. They also showed that the foliation is unique if each leaf is stable, and if it lies outside a suitable compact set. Using the unique foliation, they defined a geometric center of the foliation. Corvino and Wu \cite{CW08} proved that the geometric center of the foliation is the center of mass if the metric is conformally flat near infinity. The condition that the metric is conformally flat near infinity is later removed by the author \cite{Huang09a}.

Ye  \cite{Ye96} used a different approach to prove the existence of the foliation under the same assumption that the metric is strongly asymptotically flat, and the uniqueness of the foliation under slightly different conditions. A more general uniqueness result was proven by Qing and Tian \cite{QT07}. Metzger \cite{Metzger07} generalized the previous results to manifolds whose metrics are small perturbations of strongly asymptotically flat metrics.
However, these results have been limited to asymptotically flat manifolds with special restrictions on the $|x|^{-1}$-term of the metrics. Especially, the metric being strongly asymptotically is not coordinate invariant; namely, it no longer has the expression \eqref{SAF} if the metric is written in a boosted coordinate chart. Furthermore, center of mass is defined for asymptotically flat manifolds satisfying a more general condition: the Regge--Teitelboim condition (see Definition \ref{AF--RT}) \cite{BO87, Huang09a}, so it is desirable to generalize the previous results to this setting.

In this paper, we show that the foliation exists in the exterior region of an asymptotically flat manifold satisfying the Regge--Teitelboim condition, when the ADM mass is strictly positive. We not only remove the condition on the $|x|^{-1}$-term of the metrics, but also allow metrics to have the most general decay rates $q>1/2$. Also, we prove that the foliation is unique under certain assumptions analogous to those in \cite{HY96, Metzger07}. From our construction, the geometric center of the foliation is equal to the center of mass. To clearly state the results, we first provide some definitions.

A three-dimensional manifold $M$ with a Riemannian metric $g$ and a symmetric $(0,2)$-tensor $K$ is called an \emph{initial data set} if $g$ and $K$ satisfy the constraint equations 
\begin{align} \label{ConE}
		 R_g - | K |_g^2 + (\mbox{tr}_g ( K ) )^2 = 16 \pi \rho, \notag&&\\
		\mbox{div}_g ( K  - \mbox{tr}_g ( K ) g ) = 8 \pi J, &&
\end{align}
where $ R_g $ is the scalar curvature of $M$, $\mbox{tr}_g ( K ) = g^{ ij } K_{ ij }$, $\rho$ is the observed energy density, and $J$ is the observed momentum density. We use the Einstein summation convention and sum over repeated indices; though, sometimes we employ summation symbols for clarity.
\begin{defi} \label{AF}
$( M, g, K) $ is asymptotically flat (AF) at the decay rate $q \in (1/2,1]$ if it is an initial data set, and there exist coordinates $\{ x \}$ outside a compact set, say $B_{R_0}$, such that
\begin{align*}
		g_{ij}(x) = \delta_{ij} + O_5( |x|^{-q}) &\qquad K_{ij}(x) = O_1( |x|^{-1-q} ).
\end{align*}
Also, $\rho$ and $J$ satisfy 
\[
	\rho(x) = O(|x|^{-2-2q})\qquad J(x) = O(|x|^{-2-2q}).
\]
\end{defi}
Here, the subscript in the big $O$ notation denotes the order of the derivatives which possess the corresponding decay rates. For example, if $f=O_2(|x|^{-q})$, then $f\in C^2$ and $|f(x)| \le c |x|^{-q}$, $| Df(x)| \le c |x|^{-1-q}$, $| D^2 f(x) | \le c |x|^{-2-q}$ pointwisely for $|x|$ large, where $c$ is a constant depending only on $g$ and $K$. 
\begin{rmk}
The condition on the regularity of $g$ up to the fifth order of derivatives is used in the proof of uniqueness: Theorem \ref{GU1} and Theorem \ref{GU2}. For the existence of the constant mean curvature foliation (Theorem \ref{THM1}), we only need $g_{ij} = \delta_{ij} + O_2(|x|^{-q})$. 
\end{rmk}
For AF manifolds, the ADM mass $m$ is defined by, 
\begin{align} \label{MASS1}
		m = \frac{1}{ 16\pi } \lim_{r\rightarrow \infty} \int_{ |x| = r} \sum_{i,j} \big( g_{ij,i}- g_{ii,j} \big)\frac{x^j}{|x|} \, d\sigma_e,
\end{align}
where $|x| = \sqrt{\sum_i (x^i)^2}$, and $d\sigma_e$ is the induced area form with respect to the Euclidean metric. The ADM mass is well-defined when the decay rate $q$ is greater than $1/2$ (see \cite{Bartnik86, Chrusciel88}). Another equivalent definition of ADM mass is
\begin{eqnarray} \label{MASS2}		
		m = \frac{1}{16\pi } \lim_{r \rightarrow \infty } \int_{ |x| = r} \left( Ric^M_{ij} - \frac{1}{2} R_g g_{ij} \right) (-2x^i) \frac{x^j}{|x|} \, d \sigma_e,
\end{eqnarray}
where $Ric^M$ is the Ricci curvature of $g$. 
\begin{defi} \label{AF--RT}
$( M, g, K )$ is asymptotically flat satisfying the Regge--Teitelboim condition (AF--RT)  at the decay rate $q \in (1/2,1]$ if $ ( M, g, K ) $ is asymptotically flat, and  $g, K$ satisfy these asymptotically even/odd conditions
\begin{align*} 
g_{ij} (x) - g_{ij}(-x) = O_2 ( |x|^{-1-q} ) &\qquad K_{ij}(x) + K_{ij}(-x) = O_1 ( |x|^{-2-q} ).
\end{align*}
Also, $\rho$ and $J$ satisfy 
\[
	\rho(x) - \rho(-x) = O(|x|^{-3-2q}) \qquad J(x) - J(-x) = O(|x|^{-3-2q}).
\]
\end{defi}
\begin{rmk}
The RT condition on the data is preserved under coordinate translations, rotations, and boost. 
\end{rmk}
Assume that $( M, g, K )$ is AF--RT. Then, the center of mass $\mathcal{ C}$ is defined by, for $\alpha = 1,2,3$,  
\begin{eqnarray} \label{CTM}
	\mathcal{C}^{\alpha} &=&\frac{1}{ 16 \pi m}  \lim_{ r \rightarrow \infty} \left[ \int_{ |x| = r}  \sum_{ i, j}x^{\alpha} ( g_{ij,i} - g_{ii,j} ) \frac{x^j}{|x|} d\sigma_e \right.\nonumber\\
	&& \qquad\qquad \qquad \qquad\left.- \int_{ |x| = r} \sum_{i} \left( g_{i\alpha} \frac{x^i}{|x|} - g_{ii} \frac{x^{\alpha}}{|x|}\right) \,d\sigma_e\right].
\end{eqnarray}
The above notion is well-defined \cite{BO87, Huang09a, CD03} for AF--RT manifolds. It is noted that another notion of center of mass analogous to (\ref{MASS2}) 
has been studied and proven to be equivalent to $\mathcal{C}$ in \cite{Huang09a}. For the purpose of this paper, we use the above definition \eqref{CTM}.

We denote $S_R(\mathcal{C}) =\{ x : |x-\mathcal{C}|=R\}$ and $\nu_g$ as the outward unit normal vector on $S_R(\mathcal{C})$ with respect to $g$. If $\psi \in C^{2,\alpha} (S_R(\mathcal{C}))$, then $ \psi^*(y) := \psi (Ry+\mathcal{C})$ and $\psi^* \in C^{2,\alpha}(S_1(0))$. $(\psi^*)^{odd}$ denotes $\psi^*(y) - \psi^*(-y)$. For the definitions of {\it strictly stable} and {\it stable}, please refer to Definition \ref{STABILITY}. Also, throughout this article, $c$ and $c_i$ denote constants independent of $R$. 

Our main theorems are the following:

\begin{THM}\label{THM1}
Assume that $(M, g, K)$ is AF--RT at the decay rate $q \in  (1/2, 1]$. If $m\neq0$, then there exist surfaces $\{ \Sigma_R \}$ with constant mean curvature $H_{\Sigma_R}$  in the exterior region of $M$, and  $H_{\Sigma_R}= (2 / R) + O(R^{-1-q})$. Moreover, $\Sigma_R$ is a $c_0 R^{1-q}$-graph over $S_R (\mathcal{C})$, i.e. 
\[
	\Sigma_R =\left \{ x +  \psi_0 (x) \nu_g :  \psi_0 \in C^{2,\alpha}(S_R(\mathcal{C})) \right\}
\] 
with 
\[
	\| \psi_0^*\|_{C^{2,\alpha} (S_1 (0))} \le c_0 R^{1-q}, \quad \mbox{and} \quad \| (\psi_0^*)^{odd}\|_{C^{2,\alpha} (S_1 (0))} \le c_0 R^{-q}. 
\]
Therefore, the geometric center of $\{\Sigma_R\}$ is the center of mass $\mathcal{C}$.

Additionally, if $m > 0$, then each $\Sigma_R$ is strictly stable, and $\{ \Sigma_R \}$ form a foliation. 
\end{THM}

For one single surface $N$, we have the following uniqueness result where the minimal radius is denoted by $\underline{r} = \min \{ |x| : x \in N\}$.
\begin{THM} \label{GU1}
Assume that $(M, g, K)$ is AF--RT at the decay rate $q \in  (1/2, 1]$ and $m>0$. Then there exists $\sigma_1$ so that if $N$ has the following properties:
\begin{enumerate}
\item $N$ is topologically a sphere,
\item $N$ has constant mean curvature $H = H_{\Sigma_R}$ for some $R \ge \sigma_1$,
\item $N$ is stable,
\item $\underline{r} \ge  H^{-a}$ for some $a$ satisfying $\displaystyle \frac{5-q}{2(2+q)} < a \le 1$,
\end{enumerate}
then $N = \Sigma_R$.
\end{THM}

Notice that the topological condition $(1)$ is used in Lemma \ref{BASIC}. In Theorem \ref{GU1}, we do {\it not} assume that $N$ is a leaf of the foliation. Thus, in the region $M \setminus B_{H^{-a}} (0)$, $\Sigma_R$ is the only stable surface with constant mean curvature $H_{\Sigma_R}$. In particular, $\{ \Sigma_R\}$ is the only foliation by stable surfaces of constant mean curvature so that each leaf with  mean curvature $H$ lies in the region $M \setminus B_{H^{-a}} (0)$. It is noted that when the decay rate $q=1$, $a > 2/3$, which is exactly the restriction imposed in \cite{Metzger07} to derive the {\it a priori } estimates, but the radius  $H^{-a}$ increases as $q$ approaches to $1/2$. If we replace the condition on $\underline{r}$ by the condition that $\underline{r}$  and the maximal radius $\overline{r} = \max\{ |x| : x \in N\}$ are comparable, we derive a uniqueness result which holds outside a {\it fixed} compact set.

\begin{THM} \label{GU2}
Assume that $(M, g, K)$ is AF--RT at the decay rate $q \in  (1/2, 1]$ and $m>0$. There exist $\sigma_2$ and $c_2$ so that if $N$ has the following properties:
\begin{enumerate}
\item $N$ is topologically a sphere,
\item $N$ has constant mean curvature $H = H_{\Sigma_R}$ for some $R \ge \sigma_2$,
\item $N$ is stable,
\item $\overline{r} \le c_2 (\underline{r})^{1/a}$ for some $a$ satisfying $\displaystyle \frac{5-q}{2(2+q)} < a \le 1$,
\end{enumerate}
then $N = \Sigma_R$.
\end{THM}

An ingredient used in Section 2 (Lemma \ref{lemma1}) and hence in Theorem \ref{THM1} is the density theorem for $(M,g,K)$ satisfying the AF--RT condition. Denote the momentum tensor $\pi = K - (\mbox{tr}_g K )g$ below and denote the modified Lie derivative, for any metric $g$,
\[
		\mathcal{L}_g X := L_{X} g - \mbox{div}_{g} ( X)  g,
\]
where $L_Xg$ is the Lie derivative.  

\begin{defi} \label{HA}
$(M, g, \pi)$ is said to have harmonic asymptotics if $(M,g,\pi)$ is asymptotically flat and
\begin{align}
		g = u^4 \delta,  \quad \pi = u^2 ( \mathcal{L}_{\delta} X ) \label{Approx}
\end{align}
outside a compact set for some function $u$ and vector field $X$ tending to $1$ and $0$ at infinity respectively.
\end{defi}

\begin{defi}
We denote $W^{k,p}_{-q} (M)$ the weighted Sobolev spaces. We say that $f\in W^{k,p}_{-q}(M)$, if $f \in W_{loc}^{k,p}(M)$ and, in addition, when $p<\infty$, 
$$
		\| f\|_{W^{k,p}_{-q} (M)} := \left( \int_M \sum_{|\alpha| \le k} \left( \big| D^{\alpha} f \big| \rho ^{ | \alpha| + q }\right)^p \rho^{-3} \, d vol_g  \right)^{\frac{1}{p}} < \infty,
$$
where $\alpha$ is a multi-index and $\rho$ is a continuous function with $\rho = |x|$ on $M\setminus B_{R_0}$; when $p=\infty$, 
\[
		\| f\|_{W^{k,\infty}_{-q} (M)} := \sum_{|\alpha | \le k} ess \sup_{M } | D^{\alpha} f | \rho^{ |\alpha| + q} < \infty.
\] 
\end{defi}

\begin{THM}[Density Theorem \cite{Huang09a}]{\label{DenThm2}}
Assume that $(M,g,K)$ is AF--RT at the decay rate $q \in (1/2, 1)$. Then, there is a sequence of data $ ( \overline{ g }_k , \overline{ \pi }_k ) $ of harmonic asymptotics satisfying \eqref{ConE} (with the same $\rho$ and $J$) such that: Given any $ \epsilon > 0 $ and $q_0 \in ( 0, q)$, there exist $R $ and $k_0 = k_0(R)$ so that, for any $p>3/2$,  $( \overline{ g }_k , \overline{ \pi }_k ) $  is within an $\epsilon$-neighborhood of $ ( g, \pi ) $ in $W^{2,p }_{ - q} (M) \times W^{ 1,p }_{ -1 - q} (M) $ and
\begin{align*}
		& \| \overline{g}_k(x) - \overline{g}_k(-x) \|_{ W_{ -1 - q_0 }^{2,p} ( M \setminus B_R ) } \le \epsilon, \\
		 &\| \overline{ \pi }_k(x)+ \overline{\pi}_k(-x) \|_{  W_{ -2 - q_0 }^{1,p} ( M \setminus B_R ) } \le \epsilon, \quad \mbox{ for all } k \ge k_0.
\end{align*}
Moreover, mass, linear momentum, center of mass, angular momentum of $ ( \overline{ g }_k, \overline{ \pi }_k)$ are within $ \epsilon $ of those of $ ( g, \pi )$.
\end{THM} 
\begin{rmk}
The density theorem stated in \cite{Huang09a} is for vacuum initial data, i.e. $\rho=0$ and $J =0$. A slight modification of the proof generalizes to the current situation. Also, notice that as in \cite{Huang09a}, the theorem holds more generally for $(g,\pi)$ satisfying weaker regularity (in weighted Sobolev spaces). Here,  we only need the version that $(g,\pi)$ satisfies the pointwise regularity at the suitable decay rates defined by Definition \ref{AF--RT}. 
\end{rmk}

The article is organized as follows. In Section 2, an important identity relating the mean curvature to center of mass (\ref{CENTER}) is derived using the density theorem. In Section 3, we prove the existence of the foliation (Theorem \ref{EXIST} and Theorem \ref{FOLI}) and show its geometric center is equal to the center of mass (Corollary \ref{GCTM}).  In Section 4, Theorem \ref{GU1} and Theorem \ref{GU2} are proven after certain {\it a priori} estimates are established.


\section{Estimates on Surfaces Close to Euclidean Spheres}


This section contains three technical lemmas. Throughout this section, we assume that $(M,g,K)$ is AF--RT at the decay rate  $q\in(1/2,1]$. Denote $S_R(p) := \{ x : |x - p|= R\}$. We can view $S_R(p)$ as a submanifold in $M$ with respect to either the physical metric $g$ or the Euclidean metric $g_e$. Because $g$ is asymptotic to Euclidean metric near infinity, the induced metric on $S_R(p)$ is close to the standard spherical metric, for $R$ large. Hence, the geometric quantities on $S_R(p)$ are close to those on the standard sphere, up to the error terms. In order to construct constant mean curvature surfaces, we need to compute explicitly the leading order terms in the error terms and also estimate the rest of terms.

In the first lemma, the mean curvature of $S_R(p)$ with respect to $g$ is derived. Its mean curvature after integration with $x^{\alpha} - p^{\alpha}$ gives the difference of $p$ and center of mass $\mathcal{C}$. The estimates on the second fundamental form, Laplacian, and $Ric(\nu_g, \nu_g)$ on $S_R(p)$ are obtained in the second lemma.  The analogous estimates for surfaces close to $S_R(p) $ are derived in the third lemma.

If $f$ is a function defined on $S_R(p)$, we define $f^{odd}(x) = f(x) - f (-x + 2 p)$ and $f^{even} (x) = f (x) + f (-x + 2p)$, where $x$ and $-x+2p$ are antipodal points on $S_R(p)$.  Also, $h_{ij}$ denotes $g_{ij} - \delta_{ij}$.

\begin{lemma}\label{lemma1}
Let $H_S$ be the mean curvature of $S_R(p)$ and $d\sigma_e$ be the area form of the standard spherical metric. Then
\begin{align} \label{MC}
H_S(x)=&\frac{2}{R} + \frac{1}{2} \sum_{i,j,k} h_{ij,k}(x ) \frac{ (x^i - p^i) ( x^j - p^j )( x^k  - p^k) }{ R^3} \nonumber \\
				&+ 2 \sum_{i,j} h_{ij}(x) \frac{ ( x^i - p^i )( x^j - p^j ) }{ R^3 } - \sum_{i,j}h_{ij,i}(x ) \frac{ x^j - p^j}{R} \nonumber \\
				& + \frac{1}{2} \sum_{i,j}h_{ii,j}(x)  \frac{x^j - p^j}{ R }- \sum_i\frac{h_{ii} (x) }{R} + E_0(x), 
\end{align}
where $E_0(x) = O(R^{-1-2q})$ and $E_0^{odd}(x)= O(R^{-2-2q})$. 

For $\alpha = 1,2,3,$
\begin{align}\label{CENTER}
	\int_{S_R(p)} (x^{\alpha} - p^{\alpha})  \left( H_S - \frac{2}{R} \right) \,d\sigma_e = 8\pi m (p^{\alpha} -  \mathcal{C}^{\alpha})+ O(R^{1-2q}). 
\end{align}

\end{lemma}
\begin{proof}
Let $\nabla$ be the covariant derivative of $g$. 
\[
	H_S  = {\rm div}_{g} \nu_g,
\] 
where $\nu_g $ is the outward unit normal vector field on $S_R(p)$ with respect to $g$ and  
\[
		\nu_g= \frac{\nabla | x - p | }{ \left| \nabla | x - p| \right|_g}.	
\]
Computing directly, we have
\begin{align} \label{NOR}
				\nu_g &= \left[ 1 + \sum_{s,t}\frac{1}{2} h_{st}(x) \frac{ ( x^s - p^s )( x^t - p^t ) }{ R^2 }\right] \sum_{l} \frac{ x^l - p^l }{ R } \frac{\partial }{ \partial x^l}\notag\\
			 &\qquad \qquad\qquad \qquad - \sum_{k,l}h_{kl}(x) \frac{ x^k - p^k }{ R}  \frac{\partial }{\partial x^l} + E(x),
\end{align}
where $E(x)= O( R^{-2q})$ and $E^{even}(x)= O( R^{-1-2q})$. Then a straightforward computation gives \eqref{MC}.

To prove \eqref{CENTER}, we let $f(x) = H_S - 2/R$. First we notice that the leading order term of $f(x)$ is even and vanishes after integration with the odd function $x^{\alpha}-p^{\alpha}$. Moreover, the error term $E_0$ after integration with $(x^{\alpha} - p^{\alpha})$ is of lower order $O(R^{1-2q})$. We define, for $\alpha = 1,2,3$,
\begin{align*}
				\mathcal{I}^{\alpha}_g (R)  =  \int_{S_R(p)} ( x^{ \alpha } - p^{ \alpha } ) \left[  \frac{1}{2} \sum_{i,j,k} h_{ij,k}(x) \frac{ ( x^i - p^i )(x^j - p^j) (x^k - p^k)  }{ R^3 } \right] \, d \sigma_e.
\end{align*}
Because the asymptotically flat coordinates are not globally defined in the interior, we use the Euclidean divergence theorem in the annulus $A = \{ R \le |x-p| \le R_1\}$:
\begin{align*}
		\mathcal{I}^{\alpha}_g (R_1) -  \mathcal{I}^{\alpha}_g (R) =& \frac{1}{2} \int_A \sum_{i,j,k} \left[ h_{ij,k}(x) \frac{ ( x^j - p^j ) (x^k - p^k ) ( x^{\alpha} - p^{\alpha} ) }{ |x-p|^2 }\right]_{,i} \, d x\\
		=&  \frac{1}{2} \int_A  \sum_{i,j,k} h_{ij,k}(x) \left[ \frac{ ( x^j - p^j ) ( x^k - p^k ) ( x^{\alpha} - p^{\alpha} ) }{ |x-p|^2 } \right]_{,i} \, d x\nonumber \\
		&+\frac{1}{2} \int_A \sum_{i,j,k}\left[ h_{ij,i}(x) \frac{ ( x^j - p^j ) ( x^k - p^k ) ( x^{\alpha} - p^{\alpha} ) }{ |x-p|^2 } \right]_{,k} \, d x\nonumber \\
		& -  \frac{1}{2} \int_A \sum_{i,j,k} h_{ij,i}(x) \left[ \frac{ ( x^j - p^j ) ( x^k - p^k ) ( x^{\alpha} - p^{\alpha} ) }{ |x-p|^2 } \right]_{,k} \, d x. 
\end{align*}
Using integration by parts and simplifying the expression, we obtain an identity containing purely the boundary terms
\begin{align} \label{eq:boundary}
		\mathcal{I}^{\alpha}_g (R_1) - \mathcal{I}^{\alpha}_g (R) = \mathcal{B}^{\alpha}_g (R_1) - \mathcal{B}^{\alpha}_g (R) \quad \mbox{ for all } R_1 \ge R,
\end{align}
where $\mathcal{B}^{\alpha}_g (R)$ equals the boundary integral: 
\begin{align*}
		&\int_{S_R(p)} ( x^{\alpha} - p^{\alpha} ) \sum_{i,j} \left[\frac{1}{2} h_{ij,i}(x) \frac{ x^j - p^j }{ R } - 2 h_{ij}(x)\frac{ ( x^i - p^i)( x^j - p^j )  }{ R^3}\right] \,d \sigma_e \\
		&+ \int_{S_R(p)}\frac{1}{2}  \sum_i \left[h_{ii} (x)\frac{ x^{ \alpha } - p^{ \alpha }  }{ R} +  h_{i \alpha }(x) \frac{ x^i - p^i }{ R}  \right] d \sigma_e. 
\end{align*}
{\bf Claim}: $\mathcal{I}^{\alpha}_g(R)= \mathcal{B}^{\alpha}_g (R)$.
\begin{proof}
First notice that if $\overline{g} = u^4 \delta$ outside a compact set, then by direct computation and \eqref{eq:boundary}, for any $R_1$ large (so that $\overline{g}=u^4 \delta$ on $B_{R_1}(p)$),  
\[
	\mathcal{I}^{\alpha}_{\overline{g}}(R)- \mathcal{B}^{\alpha}_{\overline{g}} (R) = \mathcal{I}^{\alpha}_{\overline{g}}(R_1)- \mathcal{B}^{\alpha}_{\overline{g}} (R_1) = 0 \quad \mbox{ for } \alpha = 1, 2, 3.
\]
To prove the identity for general metrics, we apply Theorem \ref{DenThm2} and would like to show that, given $\epsilon_0>0$, there exists $\overline{g}$ so that, for some $R_1$,  
\begin{align} \label{eq:approximation}
	|\mathcal{I}^{\alpha}_g (R_1) -  \mathcal{B}^{\alpha}_g (R_1) | \le | \mathcal{I}^{\alpha}_{\overline{g}}(R_1) -  \mathcal{B}^{\alpha}_{\overline{g}} (R_1) |+ \epsilon_0 = \epsilon_0.
\end{align}
We denote symbolically
\[
	\int_{S_r(p)} | D (g-\overline{g}) | r \, d\sigma_e = \left| \left[\mathcal{I}^{\alpha}_g (r) -  \mathcal{B}^{\alpha}_g (r) \right] - \left[\mathcal{I}^{\alpha}_{\overline{g}}(r) -  \mathcal{B}^{\alpha}_{\overline{g}} (r)\right] \right|.
\]
Then by H\"{o}lder's inequality, 
\[
	\int_{R}^{2R} \int_{S_r(p)} | D (g - \overline{g})| r\, d\sigma_e dr \le C(g, q, p) \|g-\overline{g}\|_{W^{2,p}_{-q}(M)} R^{3-q}.
\]
That means, for a.e. $r \in (R, 2R)$, say $r=R_1$, that 
\[
	 \int_{S_{R_1}(p)} | D (g - \overline{g})| R_1\, d\sigma_e dr \le C(g, q, p) \|g-\overline{g}\|_{W^{2,p}_{-q}(M)} R^{3-q}. 
\]
Given $\epsilon = \epsilon_0/ (C(g, q, p) R^{3-q})$, there exists $\overline{g}$ so that $\|g-\overline{g}\|_{W^{2,p}_{-q}(M)} \le \epsilon$ by Theorem \ref{DenThm2}. Hence 
\[
	 \int_{S_{R_1}(p)} | D (g - \overline{g})| R_1\, d\sigma_e dr \le \epsilon_0,
\]
and then \eqref{eq:approximation} holds. Because $\epsilon_0$ is arbitrary, we prove the claim.
\end{proof}
Then, substituting $\mathcal{I}_g^{\alpha} (R)$ by $\mathcal{B}^{\alpha}_g(R)$ into \eqref{MC} and \eqref{CENTER},  and simplifying the expression, we have
\begin{align*}
		&\int_{S_R(p)} ( x^{\alpha} - p^{\alpha} ) \left( H_S - \frac{2}{R}\right) \, d\sigma_e \\
		 &= - \frac{1}{2} \left[  \int_{S_R(p)}  (x^{\alpha} - p^{\alpha} ) \sum_{i,j} ( h_{ij,i} - h_{ii,j} )\frac{ x^j - p^j }{ R } \, d\sigma_e \right.\\
		&\quad \left.-  \int_{S_R(p)} \sum_i \left(  h_{i\alpha} \frac{ x^i - p^i }{ R}- h_{ii} \frac{ x^{\alpha} - p^{ \alpha} }{ R} \right) \,d\sigma_e \right] + O(R^{1-2q}).
\end{align*}
Using the definitions of the ADM mass (\ref{MASS1}) and center of mass (\ref{CTM}), we derive \eqref{CENTER}.

\end{proof}


In the following lemmas, $c$ denotes a constant independent of $R$. Also, we denote $f^*$ to be the pullback of $f$ defined by $f^* (y) = f (Ry + p )$, so $f^*$ is a function on $S_1(0)$. Also, define
\[
	(f^*)^{odd} = f^*(y) - f^*(-y), \quad (f^*)^{even} = f^*(y) + f^*(-y).
\]

\begin{lemma}\label{lemma2} Let $A_S$ be the second fundamental form on $( S_R(p) , g_S )$ where $g_S $ is the induced metric on $S_R(p)$ from $g$, $\Delta_S$ be the Laplacian on $( S_R(p) , g_S )$, and $\nu_g$ be the outward unit normal vector. Let $\Delta_S^e$ be the standard spherical Laplacian on $S_R(p)$. Then
\begin{eqnarray*} 
&(i)		&| A_S |^2 = \frac{2}{R^2} + E_1, \quad \mbox{where } | E_1 | \le  c R^{-2-q} \mbox{ and } | E_1^{odd} | \le c R^{-3 - q}. \\
&(ii) 		&\mbox{For any } f\in C^{2,\alpha}( S_R (p) ),\\
&			& \Delta_S f = \Delta_S^e f + E_2, \quad \mbox{where } | E_2 | \le c R^{-2-q} \| f^* \|_{ C^2 (S_1(0)) } \\
&& \quad \mbox{ and } | E_2^{odd} | \le c \left( R^{-3 - q} \| f^* \|_{ C^2(S_1(0)) }+ R^{-2-q} \| (f^*)^{odd}  \|_{C^2(S_1(0)) }  \right). \\
&(iii)		& Ric^M (\nu_g, \nu_g ) = E_3,  \quad \mbox{where } | E_3| \le c R^{-2-q}\mbox{ and } | E_3^{odd} | \le c R^{-3 - q}.
\end{eqnarray*}
\end{lemma}
\begin{proof}
Let $\{ u_1, u_2\}$ be local coordinates on $S_R(p)$ and $\nabla$ be the covariant derivative of $(M,g)$.  In the rest of the section, we temporarily denote $g_{a b} = g \left( \partial_a, \partial_b \right)$ for $a, b\in \{1,2, 3\}$ where $\partial_{a} = \frac{\partial }{ \partial u_a} $ if $a  \in \{1,2\}$ and $\partial_3 = \nu_g$ (instead of the original meaning of $\{g_{ij}\}$ on the asymptotically flat coordinates in Definition {\ref{AF}}). Therefore, the second fundamental form $A_S$ is
\begin{equation}  \label{AS}
		( A_S )_{ab} = - g \left( \nabla_{\frac{\partial}{\partial u_a} } \frac{\partial}{ \partial u_b}, \nu_g \right)  = - \Gamma_{ab}^{3}.
\end{equation}
Because $g$ is asymptotically flat, $g(x) = g_e + h$ and $h = O( |x|^{-q })$. Locally, we have
\begin{equation} \label{GA}
		 \Gamma_{ab}^{3} = \frac{1}{2} \left( g_{a 3, b} + g_{b 3, a} - g_{ab, 3} \right) = g_e \left( \nabla^e_{\frac{ \partial }{\partial u_a} } \frac{ \partial }{ \partial u_b} , \nu_e\right) + |h \partial h| + | \partial h |,
\end{equation}
where we denote the difference of $ \Gamma_{ab}^3$ and $ g_e( \nabla^e_{\frac{ \partial }{\partial u_a} } \frac{ \partial }{ \partial u_b} , \nu_e) $ symbolically by $|h\partial h| + | \partial h|$, where $\nabla^e$ is the covariant derivative and the Christoffel symbols of $(M\setminus B_{R_0}, g_e)$ and $\partial$ denotes the derivative in either tangential or normal directions on $S_R(p)$.  
\begin{rmk}
More precisely, writing $f = | \partial h|$ symbolically means
\[
	|f| \le c |\partial h|, \quad  | f^{even} | \le c | (\partial h)^{even}|,\quad | f^{odd} | \le c | (\partial h)^{odd}|.
\]
The constant $c$ is independent of $R$. Notice that the derivatives in the tangential and normal directions do {\it not} affect the asymptotic even/odd property, but only improve the decay rate. For example, if $h = O( |x|^{-q})$ and $h^{odd} = O( |x|^{-1-q})$. Then $\partial h = O( |x|^{-1-q})$ and $\partial h $ is still asymptotically even at the decay rate $(\partial h )^{odd}= O( |x|^{-2-q} )$. In the following arguments, we will use similar notations to bound lower order terms for simplicity. 
\end{rmk}
The second fundamental forms are
\begin{align*}
		\left( A_S \right)_{ab} = \left( A^e \right)_{ab} + | h \partial h| +  | \partial h |.
\end{align*}
Therefore, if the principal curvature of $\left( S_R(p), g_S \right)$ are denoted by $(\lambda_{S})_i$, the above identity says:
\begin{eqnarray} \label{EIGEN}
		(\lambda_S)_i = \frac{1}{R} + | h \partial h| +  | \partial h |,
\end{eqnarray}
where $1/R$ is the principal curvature of the spheres in Euclidean space. Then
\begin{equation*}
		| A_S |^2 = (\lambda_S)_1^2 +  (\lambda_S)_2^2 = \frac{2}{R^2} + \frac{1}{R} (  | h \partial h| + |  \partial h | ) + ( | h \partial h| +  | \partial h | )^2.
\end{equation*}
We could conclude (i) by analyzing the error terms on the right hand side and by using the AF--RT condition.

Using $g = g_e + h$,  the Laplacian in the local coordinates is 
\begin{eqnarray} \label{LAP}
		\Delta_S f &=& \sqrt{g}^{-1}\frac{\partial}{ \partial u_i} \left( \sqrt{g} g^{ij} \frac{\partial }{ \partial u_j} f\right) \notag\\
		&=& \Delta^e_S f + \left( | h | | \partial g | | \partial f | + |h| |\partial^2 f| + | \partial h | | \partial f |  \right).
\end{eqnarray}
By the definition of $f^*$, $| \partial f (x) | = R^{-1} | \partial f^*(y)|$ and $| \partial^2 f(x) | = R^{-2} | \partial^2 f^*(y)|$, and then (ii) follows.
 
For (iii), notice that $Ric^M (\nu_g, \nu_g )= | D^2 g| $, where $D g$ denotes the usual derivatives of $g$ in $\{ \frac{\partial}{\partial x_i}\}$ directions as in Definition {\ref{AF}}.  Therefore, $  | D^2 g| = O \left( |x|^{-2-q} \right)$ and $\left| (D^2 g )^{odd} \right| = \left | D^2 \left(g^{odd} \right) \right | = O( |x|^{-3-q}) $.
\end{proof}

In the following lemma, we generalize the above results and prove that similar estimates also hold for surfaces which are $c R^{1-q}$-graphs over $S_R(p)$ for some constant $c$ (recall $q \in (1/2, 1]$, the decay rate of the AF metrics). Notice that when $R$ is large, the unit normal vector $\nu_g$ is close to the Euclidean normal vector, so the normal graphs over $S_R(p)$ are well-defined. 

Let $N$ be a normal graph over $S_R (p)$ defined by 
\begin{align*}		
		N = \left \{ \Psi(x)  = x + \psi \nu_g  : \psi \in C^2 \left(S_R (p) \right)\right\}. 
\end{align*}
For any  $f\in C^2(N)$, we let $\widetilde{f}(x) := f (\Psi(x))$ and $f^* := (\widetilde{f} )^*$, the pull-back function defined on $S_1(0)$. Let $\mu_g$ be the outward unit normal vector field on $N$, $A_N$ be the second fundamental form, and $\Delta_N$ be the Laplacian on $(N, g_N)$, where $g_N$ is the induced metric on $N$ by $g$.

\begin{lemma} \label{lemma3}
Assume that 
\begin{align} \label{eq:graphpsi}
	\| \psi^* \|_{ C^2 (S_1 (0) )} \le c R^{1-q} \mbox{ and } \| (\psi^*)^{odd} \|_{ C^2 (S_1 (0) )} \le c R^{-q}.
\end{align}
Then
\begin{eqnarray*} 
&(i)		&| A_N |^2 = \frac{2}{R^2} + E'_1 \quad \mbox{where } | E'_1 | \le c R^{-2-q} \mbox{ and } | (E'_1)^{odd} | \le c R^{-3- q}. \\
&(ii) 		&\mbox{For } f \in C^2(N), (\Delta_N f )( \Psi(x)) = \Delta_S^e \widetilde{f}(x) + E'_2,\\
&& \mbox{where } | E'_2 | \le c R^{-2-q} \| f^* \|_{ C^2 (S_1(0))} \mbox{ and }\\
&& | (E'_2)^{odd} | \le c\left( R^{-2 - 2q} \| f^* \|_{ C^2 (S_1(0))}+ R^{-2-q} \| (f^*)^{odd}  \|_{C^2(S_1(0)) }  \right). \\
&(iii)		&\left( Ric^M (\mu_g, \mu_g ) \right) (\Psi (x))= E'_3  \quad \\
&&\qquad \qquad \mbox{where } | E'_3| \le c R^{-2-q}\mbox{ and } | ( E'_3)^{odd} | \le c R^{-3 - q}.
\end{eqnarray*}
\end{lemma}
\begin{proof}
Similarly as in the proof of Lemma \ref{lemma2}, let $\{ u_1, u_2\} $ be local coordinates on an open set $U$ of $x\in S_R(p)$. Moreover, without loss of generality, we assume $\{\frac{\partial }{ \partial u_1}, \frac{ \partial }{ \partial u_2}, \nu_g\}$ are orthonormal at $x$ with respect to the metric $g$. Let $\{ v_1, v_2\}$ be the corresponding local coordinates on $V = \Psi(U) \subset N$ and $\mu_g$ be the outward unit normal vector field on $N$ with respect to $g$. Because $M$ is AF, up to lower order terms, we have  
\begin{align} 
		\frac{ \partial }{ \partial v_i} &= \frac{ \partial }{ \partial u_i} + (A_S)_{ij}  \psi \frac{\partial }{ \partial u_j} + \frac{ \partial \psi }{ \partial u_i} \nu_g \label{VEC1}\\
		\mu_g &= \nu_g + \psi H_S \nu_g - \sum_{i=1,2}\frac{\partial \psi }{ \partial u_i} \frac{  \partial }{ \partial u_i} \label{VEC2}
\end{align}
where we parallel transport $\left \{ \frac{ \partial }{ \partial v_1}, \frac{ \partial }{ \partial v_2}, \mu_g \right\}$ to $x$ along the unique geodesic connecting $x$ and $\Psi(x)$. In this proof, we denote 
\begin{align*}
		\overline{g}_{a b} &= g (\overline{e}_a, \overline{e}_b)\qquad \mbox{where  $\overline{e}_a = \frac{ \partial }{ \partial v_a}$ if $a \in \{ 1,2\}$ and $\overline{e}_3 = \mu_g$}, \\
		g_{a b } &= g (e_a, e_b)\qquad \mbox{where  $e_a = \frac{ \partial }{ \partial u_a}$ if $a \in \{ 1,2\}$ and $e_3 = \nu_g$},
\end{align*}
where $g_{ab}$ is defined the same as in the proof of the previous lemma. By (\ref{VEC1}) and (\ref{VEC2}), we have for $i\in \{1,2\}, a,b \in \{ 1,2,3\}$
\begin{align} \label{ESTIMATE}
		\overline{g}_{i a} =& g_{i a} + |\psi| |A_S| |g| + |\partial \psi || g |  \notag\\
		\overline{g}_{i a, b} =& g_{i a, b} + | \partial \psi ||A_S| | g| + | \psi|| \partial A_S| | g| + | \psi | |A| |\partial g| \notag\\
		& + | \partial^2 \psi | | g| + | \partial \psi|^2 | \partial g|.
\end{align} 
To prove (i), notice that
\[
		\left( A_N \right)_{ij} = - g \left( \nabla_{\frac{ \partial }{ \partial v_i}}\frac{ \partial }{ \partial v_j}\, , \mu_g  \right) = - \overline{\Gamma}_{ij}^3
\]
and
\begin{align*}
		 \overline{\Gamma}_{ab}^3 =&  \frac{1}{2} \left( \overline{g}_{a3,b} + \overline{g}_{b3,a} - \overline{g}_{ab,3}\right) \\
		 =& g\left( \nabla_{\frac{ \partial }{ \partial u_a} }\frac{\partial} {\partial u_b} , \nu_g \right)\\
		 & +  | \partial \psi ||A_S| | g| + | \psi|| \partial A_S| | g| + | \psi | |A_S | |\partial g| + | \partial^2 \psi | | g| + | \partial \psi|^2 | \partial g|. 
\end{align*}
Therefore, by (\ref{EIGEN}) and the previous two identities, we get
\begin{align*}
		| A_N |^2 = & |A_S|^2+ \frac{1}{R} \left( | \partial \psi ||A_S| | g| + | \psi|| \partial A_S| | g|\right. \\
		&\left. + | \psi | |A| |\partial g| + | \partial^2 \psi | | g| + | \partial \psi|^2 | \partial g| \right).
\end{align*}
Above, the terms of the weakest decay rate in the error terms are, for instance,
\begin{align*}
		\frac{1}{R}  | \partial^2 \psi | | g| = O( R^{ -2- q }).
\end{align*}
Similarly, we could compute $(E_1')^{odd}$ and use Lemma \ref{lemma2}(i) to conclude (i). Moreover, we can derive from the above two identities to conclude that the trace-free second fundamental form is
\begin{align} \label{TraceFree}
			| \mathring{A}_N | = O(R^{-1-q}),
\end{align}
and the mean curvature of $N$ is
\begin{align} \label{eq:graphmc}
	H_N = \frac{2}{R} + O(R^{-1-q}).
\end{align}

For (ii), the Laplacian in local coordinates is
\begin{align*}
		(\Delta_{N} f)(\Psi(x)) =& \sqrt{\bar{g}}^{-1} \frac{\partial }{ \partial v_i} \left(\sqrt{ \bar{g} } \bar{g}^{ij} \frac{\partial }{ \partial v_j} f(\Psi(x))\right) \\
		=& \sqrt{g}^{-1} \frac{\partial }{ \partial u_i} \left(\sqrt{ g } g^{ij} \frac{\partial }{ \partial u_j} f(\Psi(x))\right) + | \partial \psi | |A_S| |g| | \partial f| \\
		&+ | \psi | | \partial A_S| |g| | \partial f | + | \psi | |A_S| |g| | \partial g| | \partial f| + | \psi | |A_S| |g| | \partial^2 f|.
\end{align*}
Then
\begin{align*}
		&(\Delta_N f )( \Psi(x)) = \Delta_S \widetilde{f}(x) + | \partial \psi | |A_S| |g| | \partial \widetilde{f}| \\
		&\qquad \qquad+ | \psi | | \partial A_S| |g| | \partial \widetilde{f} | + | \psi | |A_S| |g| | \partial g| | \partial \widetilde{f}| + | \psi | |A_S| |g| | \partial^2 \widetilde{f}|,
\end{align*}
where the terms at the weakest decay rate of the error terms are, for instance,
\begin{align*}
		 |\partial \psi | | A_S| |g| | \partial \widetilde{f}(x)| \le  R^{-1} |\partial \psi | | A_S| |g| | \partial f^*(x)| \le CR^{-2-q} \| f^*\|_{C^{2,\alpha}}.
\end{align*}
Then, (ii) follows from Lemma \ref{lemma2} (ii).

Using Lemma {\ref{lemma2}}(iii) and the identity
\begin{align*}
		Ric^M (\mu_g, \mu_g) (\Psi(x)) = Ric^M(\nu_g, \nu_g) + | D^2 g|| \psi | |A_S| + |D^2 g| |\partial \psi|,
\end{align*}
we can conclude (iii).

\end{proof}

\section{Existence of the Foliation}
In this section, we prove the existence of the foliation of constant mean curvature surfaces, assuming the ADM mass $ m > 0$. An idea similar to  \cite{Ye96} is employed in which normal perturbations of Euclidean spheres are considered. However, our construction is more subtle because we have to perturb a Euclidean sphere $S_R(p)$ \emph{twice} to construct a constant mean curvature surface. Roughly speaking, the first perturbation is of the order $O(R^{1-q})$ and the second one is of the order $O(R^{1-2q})$. Geometrically, it reflects the fact that, under weaker asymptotics, constant mean curvature surfaces are too far away from some $S_R(p)$ to apply the implicit function theorem directly. Therefore, we have to construct a family of \emph{approximate spheres} $\mathcal{S}(p,R)$ from $S_R(p)$ using a PDE construction. Then by carefully choosing the center $p$, we find the nearby constant mean curvature surfaces from $\mathcal{S}(p,R)$.

While we only require $ m \neq 0$ in proving Theorem \ref{EXIST}, assuming $m > 0 $ is used to prove the stability of the surfaces and then to show that they form a foliation. From our construction, each leaf of the foliation is a graph over the Euclidean sphere centered at some $p= p(R)$. We also show that $p$ converges to the center of mass $\mathcal{C}$ as $R \rightarrow \infty$. 

Throughout this section, $c = c(\alpha, g, \partial g)$ or $c_i = c_i (\alpha, g, \partial g)$ denote constants independent of $R$. Recall that if $\psi \in C^{2,\alpha} (S_R(p))$, then $ \psi^*(y) := \psi (Ry+p)$ and $\psi^* \in C^{2,\alpha}(S_1(0))$, and define 
\[
	(\psi^*)^{odd} = \psi^*(y) - \psi^*(-y), \quad (\psi^*)^{even} = \psi^*(y) + \psi^*(-y).
\]
The first theorem states the existence of a surface with the given constant mean curvature. 

\begin{thm}\label{EXIST}
Assume that $(M,g, K)$ is AF--RT with $q \in (1/2, 1]$ and $m\neq 0$. There exist constants $\sigma_0$ and $c_0$ so that, for all $R > \sigma_0$, there is $\Sigma_R$ with constant mean curvature 
\[
	H_{\Sigma_R} = \frac{2}{R} + O(R^{-1-q}).
\]
$\Sigma_R$ is  a $c_0 R^{1-q}$-graph over $S_R(p)$, i.e. 
\[		
	\Sigma_R =\left \{ x +  \psi_0 \nu_g : \psi_0 \in C^{2,\alpha} (S_R( p)) \right\}
\]
and $\psi_0$ satisfies
\begin{align} \label{eq:psi0}
	\| \psi_0^*\|_{C^{2,\alpha} (S_1 (0))} \le c_0 R^{1-q}, \quad \mbox{and}\quad \| (\psi_0^*)^{odd} \|_{C^{2, \alpha} (S_1(0))} \le c_0R^{-q}. 
\end{align}
\end{thm}

Because the mean curvature of $S_R(p)$ is equal to $2/R$ up to $O( R^{-1-q})$-terms \eqref{MC}, we would like to construct a constant mean curvature surface by perturbing $S_R(p)$ in the normal direction. However, in contrast to the case that $(M,g)$ is strongly asymptotically flat, the mean curvature of $S_R(p)$ is {\it not} close to some constant enough to apply the implicit function theorem. Therefore, we first construct the unique approximate spheres $\mathcal{S}(p,R)$ associated to $S_R(p)$ whose mean curvature is closer to some constant up to $O( R^{-1-2q})$-terms.

Recall that $f$ denotes $H_S - 2/R$, and $H_S$ is the mean curvature of $S_R(p)$. 

\begin{lemma} \label{lemma:approx}
There exists $c$ independent of $R$ so that, for $R$ large, there is an approximate sphere 
\begin{align*}
	\mathcal{S}(p,R) = \left\{ x + \phi (x) \nu_g : \phi \in C^{2,\alpha} (S_R(p))\right\}, 
\end{align*}
where $\phi$ satisfies
\begin{align} \label{PHI} 
	\| \phi^* \|_{C^{2,\alpha} (S_1(0))} \le c R^{1-q}, \quad \| (\phi^*)^{odd}\|_{C^{2,\alpha} (S_1(0))} \le c R^{-q}.
\end{align}
Moreover, the mean curvature of $\mathcal{S}(p,R)$ is
\begin{align} \label{eq:approxmean}
	H_{\mathcal{S}} = \frac{2}{R} + \overline{f} + O(R^{-1-2q}),
\end{align}
where $\overline{f} := (4 \pi R^2)^{-1}\int_{S_R(p)} f \, d\sigma_e$.
\end{lemma}
\begin{rmk}
When $q=1$, $\phi$ is bounded by a constant. However, when $q<1$, the size of $\phi$ may increase as $R$ increases. 
\end{rmk}
\begin{proof}
$L_0 = - \Delta_S^e - 2/R^2$ denotes the linearized mean curvature operator on the standard sphere $S_R(p)$ in Euclidean space , where $\Delta_S^e$ is the standard spherical Laplacian. It is known that  because mean curvature is preserved by translations in the Euclidean space, $L_0$ has the kernel  
\[
	\mathfrak{K} = \mbox{span} \{ x^1-p^1, x^2 - p^2, x^3-p^3\}.
\] 
Also notice that, by the self-adjointness of $L_0$, the $L^2$ orthogonal complement 
\[
	\mathfrak{K}^{\perp} = \mbox{Range} L_0.
\] 
 Let $L_0 : C^{2,\alpha}(S_R(p)) \rightarrow C^{0,\alpha}(S_R(p))$. Consider
\begin{align}
		L_0 \phi = f - R^{-3-q} \sum_{i} A^i (x^i-p^i) -  \overline{f}.  \label{APPROX}
\end{align}
We choose the constants $A^i$ to satisfy
\begin{align}\label{eq:A}
	A^i = \frac{3}{4\pi} R^{-1+q} \int_{S_R(p)} (x^i - p^i) f(x) \, d\sigma_e,
\end{align}
so the right hand side of \eqref{APPROX} is in Range$L_0$ and then \eqref{APPROX} is solvable. Notice that because of the AF--RT condition, $A^i = O(1)$. We let $\phi$ be the unique solution in $\mathfrak{K}^{\perp}$ to the equation \eqref{APPROX}. 

To estimate $\phi^*$, note that it satisfies 
\[
		(- \Delta_0 - 2) \phi^* = R^2 ( f^*(y) - R^{-2-q} \sum_i A^i y^i  - \overline{f}), 
\]
where $\Delta_0$ is the standard spherical Laplacian of the unit sphere in Euclidean space. Because $\phi^* \in (\mbox{Ker}(- \Delta_0 - 2) )^{\perp}$, by the Schauder estimate and because $f = | Dh| =  O(R^{-1-q})$, 
\[
		\| \phi^* \|_{C^{2, \alpha}(S_1(0)) } \le c \|  R^2 f^*(y) - R^{-q}\sum_i A^iy^i  -  R^2\overline{f}  \|_{C^{0, \alpha}(S_1(0))} \le c R^{1-q}.
\]
Moreover, $(\phi^*)^{odd}$ satisfies the following equation
\begin{align*}
		(- \Delta_0 - 2) (\phi^*)^{odd} = R^2 ( f^* )^{odd} -2 R^{-q} \sum_i A^iy^i.
\end{align*}
Then, because $ (\phi^*)^{odd} \in \left( \mbox{Ker}(-\Delta_0 -2) \right)^{\perp}$, by the Schauder estimates and the fact that $f$ is asymptotically even with $f^{odd} = O(R^{-2-q})$, we have
\begin{align} 
		\| (\phi^*)^{odd} \|_{C^{2, \alpha} (S_1(0)) } &\le c  \|  R^2 ( f^*)^{odd} - R^{-q} \sum_i A^iy^i  \|_{C^{0, \alpha}(S_1(0))} \le c R^{-q}.
\end{align}
Then we define 
\[
		\mathcal{S}(p,R) = \left\{ x + \phi  \nu_g\right\}.
\]
In particular, $\mathcal{S}(p,R)$ is a graph over $ S_R ( p )$ which satisfies the conditions for $N$ in Lemma \ref{lemma3}.

We compute the mean curvature of $\mathcal{S}(p,R)$. Denoting $H_S$ the mean curvature map
\[
	H_S : C^{2,\alpha}(S_R(p)) \rightarrow C^{0,\alpha} (S_R(p))
\]
which maps a function $\phi$ to the mean curvature of the normal graph of $\phi$ over $S_R(p)$. Then the mean curvature of $\mathcal{S}(p,R)$ is $H_S(\phi)$. By Taylor's theorem,
\begin{align*}
	H_S(\phi )= H_S( 0) - L_S \phi + \int_0^1 \left(d H_S \left (s \phi \right) - d H_S (0) \right) \phi \,ds,
\end{align*}
where $ d H_S$ is the first Fr{\'e}chet derivative in the $\phi$-component, and $L_S$ is the linearized mean curvature operator on $S_R(p) $ defined by
$$
		L_S =- \Delta_S - |A_S |^2 - Ric^M(\nu_g, \nu_g),
$$
where $\Delta_{S}, A_S$, and $Ric^M(\nu_g, \nu_g) $ are defined by Lemma \ref{lemma2}. The integral term above can be bounded by $ \sup_{s \in [0,1]} \left| d^2 H_S (s\phi) \phi \phi\right|$ by the mean value inequality, and
$$
		d^2 H_S ( s \phi) \phi \phi = \left.\frac{ \partial^2  }{ \partial t^2} H_S(t\phi) \right|_{ t= s}.
$$
The left hand side is the second Fr{\'e}chet derivative and the right hand side is the second derivative of the mean curvature of the surface 
\[
	N_s := \left \{x + s\phi(x)  \nu_g : y\in S_R(p) \right\}.
\] 
For $R$ large, the unit outward normal vector field on $N_s$ is close to $\nu_g$, and a straightforward calculation gives us 
\begin{align} \label{ERRINT}
		\left|\frac{ \partial^2  }{ \partial t^2} H_S(t\phi) \right | \le& c \left( \left|R_{ijkl} \right| \left|A_{N_s} \right| | \phi|^2 + |A_{N_s}| |\partial \phi|^2 + \left | A_{N_s} \right| | \phi|  | \partial^2 \phi| + \left| A_{N_s} \right|^3 | \phi |^2\right) \notag\\
		\le& c R^{-3} \| \phi^* \|^2_{C^{2 }(S_1(0)) }.
\end{align} 
In the last inequality, we use that $|R_{ijkl}| = O(R^{-2-q})$ and $|A_{N_s} | = O(R^{-1})$ from Lemma \ref{lemma3}.

Noticing that $H_S(0)$ is the mean curvature of $S_R(p)$, so, by Lemma \ref{lemma1}, we have
\begin{align*}
	H_S(\phi ) &= \frac{2}{R} +  f (x)- L_0 \phi +( L_0 - L_S ) \phi + \int_0^1\left(d H_S \left (s \phi \right) - d H_S (0\right) \phi  \, ds. 
\end{align*}
By \eqref{APPROX}, 
\begin{align} \label{MCAPPROX}
	H_S(\phi ) = \frac{2}{R}+ \overline{f} + R^{-3-q} \sum_i A^i (x^i-p^i) + E_4,
\end{align}
where 
\begin{align*}
		E_4 =& ( L_0 - L_S) \phi 
		+ \int_0^1 \left(d H_S \left (s \phi \right) - d H_S (0) \right) \phi \, ds.
\end{align*}
By Lemma \ref{lemma2}, \eqref{PHI}, and (\ref{ERRINT}), the error term $E_4$ is bounded by
\begin{align} \label{E4}
		\| E_4^* \|_{C^{0,\alpha}} \le&  c \left( R^{-q} \| \phi^* \|_{C^{2,\alpha} }  + R^{-1} \| \phi^* \|^2_{C^{2,\alpha}} \right)\le c R^{1- 2q} \notag\\
		\| (E_4^*)^{odd} \|_{C^{0,\alpha} } \le& c \left( R^{-1-q} \| \phi^* \|_{C^{2,\alpha} }+ R^{-q} \| ( \phi^*)^{odd}\|_{C^{2,\alpha} }\right.\notag \\
		&\left.+ R^{-2} \| \phi^*\|^2_{C^{2,\alpha} } + R^{-1} \| (\phi^*)^{odd} \|_{C^{2,\alpha} } \| \phi \|_{C^{2, \alpha} }\right) \notag \\
		\le &c R^{-2q}.
\end{align}
Therefore, we derive \eqref{eq:approxmean}.
\end{proof}

\begin{pfe}
To construct a surface $\Sigma_R$ with constant mean curvature, we consider the normal perturbations on $\mathcal{S}(p,R) := \{ \Psi(x) = x + \phi \nu_g\}$. We denote the mean curvature of the normal graph $\psi$ over $\mathcal{S}(p,R)$ by $H_{\mathcal{S}} (\psi )$. By Taylor's theorem, for any $\psi \in C^{2, \alpha} (\mathcal{S}(p,R))$,
\begin{align} \label{MAIN}
	H_{\mathcal{S}}(\psi ) =& H_{\mathcal{S}} ( 0) + \Delta_{\mathcal{S}} \psi + \left( | A_{\mathcal{S}} |^2 + Ric^M (\mu_g, \mu_g)\right) \psi \nonumber\\ 
	&+ \int_0^1 \left( d H_{\mathcal{S}} ( s \psi ) - d H_{\mathcal{S} } ( 0)\right) \psi \, ds,
\end{align}
where $\Delta_{\mathcal{S}}, A_{\mathcal{S}},$ and $\mu_g$ are defined as in Lemma \ref{lemma3} for which we let $N = \mathcal{S}(p,R)$, and $\widetilde{\psi}$ and $\psi^*$ denote the pull-back functions on $S_R(p)$ and $S_1(0)$ respectively. By \eqref{MCAPPROX} and \eqref{MAIN}, 
solving 
\begin{eqnarray} \label{CONSTMEAN}
		H_{\mathcal{S}}(\psi ) = \frac{2}{R} + \overline{f}
\end{eqnarray} 
is equivalent to solving $\psi$ to the following equation:
\begin{align*} 
		0 =&  R^{-3-q}\sum_i A^i (x^i-p^i) + E_4 + \Delta_{\mathcal{S}} \psi  +\left( | A_{\mathcal{S}} |^2 + Ric^M (\mu_g, \mu_g)\right) \psi \nonumber\\ 
	&+ \int_0^1 \left( d H_{\mathcal{S}} (s \psi ) - d H_{\mathcal{S} } (0)\right)\psi \, ds.
\end{align*}
That is, to solve
\begin{align}\label{MCCONST}
	L_0 \widetilde{\psi} =R^{-3-q} \sum_i  A^i (x^i-p^i)  + E_4 + E_5,
\end{align}
where
\begin{align*}	
	E_5 (x) =&(\Delta_{\mathcal{S}} \psi)\circ \Psi(x) -  \Delta_S^e \widetilde{\psi}\\
	&+ \left( | A_{\mathcal{S}} |^2 (\Psi(x)) - \frac{2}{R^2} + (Ric^M (\mu_g, \mu_g))\circ \Psi(x) \right) \widetilde{\psi} \nonumber\\ 
	&+ \int_0^1 [\left( d H_{\mathcal{S}} ( s\psi ) - d H_{\mathcal{S} } (0)\right)  \psi ] \circ \Psi(x) \, ds.
\end{align*}
Using Lemma \ref{lemma3} and (\ref{ERRINT}), we have 
\begin{align} \label{eq:E5}
		\| E_5^* \|_{C^{0,\alpha}} \le& c \left( R^{-q} \| \psi^* \|_{C^{2,\alpha} } +   R^{-1} \| \psi^* \|^2_{C^{2,\alpha} } \right) \notag\\
		\| (E_5^*)^{odd} \|_{C^{0,\alpha}} \le& c \left(  R^{-2q} \| \psi^*\|_{C^{2,\alpha}} +  R^{-q} \left\| (\psi^*)^{odd} \right\|_{C^{2,\alpha}} \right.\notag \\
		& \left.+ R^{-2} \| \psi^*\|^2_{C^{2,\alpha}} + R^{-1}  \left\| \psi^*\right\|_{C^{2,\alpha}} \left\| (\psi^*)^{odd}\right\|_{C^{2,\alpha}}   \right).
\end{align}
We pull back (\ref{MCCONST}) on $S_1(0)$,
\begin{align} \label{eq:Main}
		(-\Delta_0 - 2) \psi^* = R^2\left( R^{-2-q}\sum_i A^i y^i+ E^*_4 + E_5^*\right)=:  F(p, R, \psi^*).
\end{align} 
If $\| \psi^*\|_{C^{2,\alpha}} \le 1$, then by \eqref{E4} and \eqref{eq:E5}, 
\begin{align} \label{eq:potential}
	&\| F(p, R, \psi^*) \|_{C^{0, \alpha}(S_1(0))} \le cR^{1-2q}, \notag\\
	& \| (F(p, R, \psi^*))^{odd} \|_{C^{0, \alpha}(S_1(0))} \le cR^{-2q}.
\end{align}
In order to find a solution $\psi^*$ to the above equation, a necessary condition is that $F(p, R, \psi^*)$ lies inside Range$(-\Delta_0 - 2)$. Using $m\neq 0$, we show this can be achieved by correctly choosing $p =  p (R, \psi^*)$. By the definition of $A^i$ (\ref{eq:A}), we have
\begin{align}\label{eq:range}
		&\int_{S_1(0)} y^{\alpha} R^{-2} F(p, R, \psi^*) \, d\sigma_e \notag \\
		= &\int_{S_1(0)}  y^{\alpha}\left( R^{-2-q} \sum_i A^i y^i+ E^*_4 + E_5^*\right) \, d\sigma_e \notag \\
		= &\int_{S_1(0)} y^{\alpha} f^* (y) \, d\sigma_e + \int_{S_1(0)} y^{\alpha} \left( E^*_4 + E_5^*\right) \, d\sigma_e \notag\\
		= & \int_{S_R(p)} \frac{x^{\alpha} - p^{\alpha }}{ R } f(x) R^{-2} \, d\sigma_e  + \int_{S_1(0)} y^{\alpha} \left( E^*_4 + E_5^*\right) \, d\sigma_e. 
\end{align}
Using (\ref{CENTER}) in Lemma {\ref{lemma1}}, the first integral is equal to 
\begin{align*}
		 8\pi m \left( p^{\alpha} - \mathcal{C}^{\alpha}\right) R^{-3} + O(R^{-2-2q}).
\end{align*}
Therefore, because $m\neq 0$, we can choose 
\begin{eqnarray} \label{PCTM}
		p^{\alpha} ( R, \psi^*)= \mathcal{C}^{\alpha}  - \frac{R^3}{8\pi m } \int_{S_1(0)} y^{\alpha} \left( E^*_4 + E_5^*\right) \, d\sigma_e + O(R^{1-2q}),
\end{eqnarray}
such that \eqref{eq:range} is zero; that is, 
\[
	F(p(R, \psi^*), R, \psi^*)\in \mbox{Range}(-\Delta_0 - 2).
\]

To complete the proof, we apply the Schauder fixed point theorem. Although $F(p(R, \psi^*), R, \psi^*)$ contains also the second order derivatives of $\psi^*$ from the error term $E_5$, those second derivatives are quasi-linear and have small coefficients. We can rewrite \eqref{eq:Main} as, for $R$ large, 
\[
	(-\Delta_0 - 2) \psi^* + a_{ij}(x, \psi^*) \partial^2_{ij} \psi^* = \overline{F}(p(R, \psi^*), R, \psi^*),
\]
where $a_{ij} = O(R^{-q})$ if $\| \psi^* \|_{C^0} \le 1$. Therefore, 
\[	
L:=(-\Delta_0 - 2)  + a_{ij}(y, \phi^*) \partial^2_{ij}
\]
 is a quasi-linear elliptic operator for $R$ large. 

Define $\mathcal{B} = C^{2, \alpha} ( S_1(0) ) \cap \{ v : \| v\|_{C^{2, \alpha} (S_1(0)) } \le1 \}$. Let $T: \mathcal{B} \rightarrow C^{2, \alpha} (S_1(0))$ be $T(v) = u$ where $u$ is the unique solution in $(\mbox{Ker}L)^{\perp}$ to the linear equation
$$
		L u = \overline{F}(p ( R, v), R, v).
$$
By the Schauder estimates and  (\ref{eq:potential}),
\begin{align*}
				&\| u \|_{C^{2,\alpha} (S_1(0) ) }\le c \| \overline{F}\|_{C^{0,\alpha} (S_1(0))}\le c  R^{1-2q}. 
\end{align*}
For $R$ large enough, the right hand side is less than $1$, so $T$ is a map from $\mathcal{B}$ to itself. It is easy to check that $T$ is compact and continuous by the standard linear theory. Therefore, the Schauder fixed point theorem applies, and there is a fixed point $\psi^*$ to (\ref{eq:Main}). Using the Schauder estimates and (\ref{eq:Main}) to $(\psi^*)^{odd}$, 
$$
		L_0 ( \psi^* )^{odd}= F^{odd}.
$$
Therefore,
\begin{align}\label{eq:psi}
	&\| \psi^* \|_{C^{2,\alpha} (S_1(0) ) }\le  \frac{1}{2} c_0  R^{1-2q}, \notag\\
	&\| (\psi^*)^{odd} \|_{C^{2,\alpha}(S_1(0)) } \le c \| F^{odd} \|_{C^{0,\alpha} (S_1(0))} \le \frac{1}{2} c_0 R^{-q}.
\end{align}

By letting $\widetilde{\psi} (x) = \psi^* \left( \frac{ x-p }{ R }\right)$, $ \widetilde{\psi}$ is a solution to the identity (\ref{MCCONST}). We let $\psi(\Psi(x)) = \widetilde{\psi}(x)$, then the graph of $\psi$ over $\mathcal{S}(p, R)$ has constant mean curvature $ ( 2/R )+ \overline{f}$. Because $\mu_g$ is close to $\nu_g$ by (\ref{VEC2}),
we can rearrange and write $\Sigma_R$ as a graph over $S_R (p)$
\begin{align*}
		\Sigma_R = \left \{ x + \psi_0 \nu_g : \psi_0 \in C^{2, \alpha} (S_R (p )) \right\}.
\end{align*}
Because $\psi_0 = \phi + \widetilde{\psi}  + O(R^{-q})$, by \eqref{PHI} and \eqref{eq:psi}, we derive \eqref{eq:psi0}.
\qed
\end{pfe} 
In \cite{HY96}, the geometric center of a constant mean curvature foliation is defined:
\begin{defi}
Let $\{ \Sigma_R \}$ be the family of surfaces constructed in the previous theorem and $X$ be the position vector.  The geometric center of mass of $(M, g, K)$ is defined by, for $\alpha  = 1,2,3$, 
\begin{align*}
		\mathcal{C}^{\alpha}_{HY} = \lim_{R\rightarrow \infty}\frac{\int_{\Sigma_R} X^{\alpha} \, d\sigma_e }{ \int_{ \Sigma_R} \, d\sigma_e}.
\end{align*}
\end{defi}
From our construction, we not only prove that the geometric center converges, but we also show that it is equal to center of mass $\mathcal{C}$. The following corollary generalizes the results in \cite{CW08, Huang09a}.
\begin{cor} \label{GCTM}
Assume $(M,g, K)$ is AF--RT at the decay rate $q \in ( 1/2, 1]$ and $m\neq 0$. Then $\mathcal{C}_{HY}$ converges and is equal to $\mathcal{C}$.
\end{cor}
\begin{proof}
Let $\Phi$ be the diffeomorphism from $S_R(p) $ to $\Sigma_R$ defined by $\Phi (x) = x + \psi_0 \nu_g$. Then by the definition and the area formula,
\begin{align*}
		\frac{\int_{\Sigma_R} X^{\alpha} \, d\sigma_e }{ \int_{ \Sigma_R} \, d\sigma_e } =& \frac{\int_{S_R(p) } \left(  x^{\alpha} + \psi_0  \nu_g^{\alpha} \right)  J\Phi \, d\sigma_e }{ \int_{ S_R (p)  }J\Phi \, d\sigma_e } \\
		=& p^{\alpha} + \frac{ \int_{S_R(p)} O(R^{1-2q}) \, d\sigma_e}{ \int_{S_R(p)} \, d\sigma_e}.
\end{align*}
In the second identity, we use \eqref{eq:psi0} and $J \Phi = 1 + O(R^{-q})$, so the second term in the last line is of lower order and vanishes after taking limits. We only need to study the limit of $p$. By (\ref{PCTM}), we estimate the error terms $E^*_4$ and $E^*_5$ in \eqref{PCTM}. By the asymptotically even/odd properties of $E_4^*$ in (\ref{E4}),  
\begin{align*}
		\left| \int_{S_1(0)} y^{\alpha} E_4^* \, d\sigma_e \right| \le& \left| \int_{S_1(0) } y^{\alpha} ( E_4^*)^{odd}  \, d\sigma_e \right|  \\
  	\le& \left| \int_{S_R(p)} \frac{( x^{\alpha} - p^{\alpha} )}{R} ( E_4)^{odd} R^{-2} \, d\sigma_e \right|\\
		\le& c R^{-2}\sup_{S_R (p) } \left| (E_4)^{odd} \right| \le c R^{-2-2q}.
\end{align*}
Similarly, by (\ref{eq:E5}) and \eqref{eq:psi},
\begin{align}  \label{E5}
		\left| \int_{ S_1 (0) } y^{\alpha}E_5^* \, d\sigma_e \right| \le& c R^{-2} \sup_{S_R(p) } \left|E_5^{odd} \right|  \le c R^{-2-2q}.
\end{align}
From \eqref{PCTM}, we derive
\[
	p^{\alpha} = \mathcal{C}^{\alpha} + O(R^{1-2q}).
\]
After taking limits, we prove the corollary.
\end{proof}

Because $p$ is asymptotic to $\mathcal{C}$, we can rearrange $\Sigma_R$ to be graphs over $S_R (\mathcal{C})$.
\begin{cor} \label{cor:graph}
The constant mean curvature surfaces $\Sigma_R$ constructed in Theorem \ref{EXIST} are $c_0R^{1-q}$-graph over $S_R(\mathcal{C})$, i.e. 
\[
	\Sigma_R =\left \{ x +  \psi_0 \nu_g: \psi_0 \in C^{2, \alpha} (S_R(\mathcal{C}))\right \},
\]
and $\psi_0$ satisfies
\[
	\| \psi_0^*\|_{C^{2,\alpha} (S_1 (0))} \le c_0 R^{1-q}, \quad \| (\psi_0^*)^{odd}\|_{C^{2,\alpha} (S_1 (0))} \le c_0 R^{-q}.
\]
\end{cor}

After constructing the family of surfaces with constant mean curvature $\{ \Sigma_R \}$, we prove that they form a smooth foliation. We first estimate the eigenvalues of the linearized mean curvature operator.

\begin{defi}\label{STABILITY}
A smooth hypersurface $N$ in $M$ is called stable if the linearized mean curvature operator 
\[
	L_N := - \Delta_N - \left( |A_N|^2 + Ric^M (\mu_g, \mu_g)\right)
\]
 has the lowest eigenvalue $\mu_0 \ge 0$ among functions with zero mean value, i.e. 
\begin{align} \label{def:mu0}
		\mu_0 : = \inf \left\{ \int_N uL_N u \, d\sigma: \| u \|_{L^2(N)} = 1, \int_N u\, d\sigma = 0,\mbox{ and  } u\not\equiv 0 \right\} \ge 0.
\end{align}
If $\mu_0$ is strictly positive, $N$ is called strictly stable.
\end{defi}
\begin{rmk}
If $N$ has constant mean curvature, $N$ being stable means that $N$ locally minimizes area among surfaces containing the same volume. The following two lemmas hold for more general surfaces. 
\end{rmk}


\begin{lemma} \label{STABLE}
Assume that $(M, g, K)$ is AF--RT at the decay rate $q\in(1/2,1]$ and $m>0$. Let $N$ be a normal graph of $\psi$ over $S_R(p)$:
\[	
		N = \left \{ \Psi(x)  = x + \psi \nu_g  : \psi \in C^2 \left(S_R (p) \right)\right\}, 
\]
where $\psi$ satisfies \eqref{eq:graphpsi} in Lemma \ref{lemma3}.  For $R$ large, $N$ is strictly stable and the lowest eigenvalue 
\[
	\mu_0\ge \frac {6m}{R^3 }+ O(R^{-2-2q}).
\]
\end{lemma}
\begin{proof}
Let $L_0 = -\Delta_S^{e} - \frac{2}{R^2}$ be the linearized mean curvature operator of standard spheres of radius $R$ in Euclidean space. $L_0$ has kernel $\mathfrak{K}$:
\[ 
	\mathfrak{K} =  \mbox{span}\left\{ \frac{x^1 - p^1}{R^2}, \frac{x^2 - p^2}{R^2}, \frac{x^3 - p^3}{R^2}\right\}.
\] 
By Lemma \ref{lemma3} and recall that $\widetilde{u}(x) = u (\Psi(x)) \in C^{2}(S_R(p))$,  for any $u\in C^{2}(N)$, 
\begin{align*}
	\left| ( L_N u )\left( \Psi (x) \right) -\left( - \Delta_S^e \widetilde{u} - \frac{2}{R^2} \widetilde{u} \right) \right| \le c   R^{-2-q} \| u^*\|_{C^2 (S_1(0)) }.
\end{align*} 
Now, we normalize $u$ to satisfy $\| u \|_{L^2(N)} = 1$, and it implies $u = O(R^{-1})$ because the area $|N| = 4\pi R^2 + O(R^{2-q})$. Then by the area formula and \eqref{eq:graphpsi},
\begin{align*}
	\int_N uL_N u \, d\sigma &= \int_N uL_N u \, d\sigma_e + O(R^{-2-2q})\\
	& =\int_{S_R(p)} \widetilde{u} (L_N u)(\Psi(x)) J\Psi \, d\sigma_e + O(R^{-2-2q})\\
	& = \int_{S_R(p)}  \widetilde{u} L_0 \widetilde{u} \, d\sigma_e +  O(R^{-2-2q}).
\end{align*}
Therefore, the infimum of \eqref{def:mu0} is achieved by $u$ satisfying $\widetilde{u} \in \mathfrak{K}$, up to lower order terms. We {\bf claim} that, if $u$ satisfies that $\widetilde{u} \in \mathfrak{K}$ and $\| u \|_{L^2(N)} = 1$, then
\begin{align} 
	\int_N uL_N u \, d\sigma \ge -\frac{3}{4\pi R^2} \int_N Ric^M(\mu_g, \mu_g) \, d\sigma + O(R^{-2-2q}). 
\end{align}
Let $u$ satisfy the assumption of the claim. By Lemma \ref{lemma3}, 
\begin{align} \label{eq:u}
	- \Delta_N u = \frac{2}{R^2} u + E'_2,
\end{align}
where
\begin{align} \label{eq:E_2}
	 \widetilde{E'_2} = O( R^{-3-q}), \quad \mbox{ and } \quad \widetilde{E'_2}^{odd} = O(R^{-3-q}).
\end{align}
Multiply \eqref{eq:u} by u and integrate over $N$. Because $\widetilde{u}$ is an odd function with respect to the center $p$,
\begin{align} \label{eq:rough}
	\int_N | \nabla^N u|^2 \, d\sigma = \frac{2}{R^2} + O(R^{-2-q}).
\end{align}
Then notice that, by \eqref{TraceFree} and \eqref{eq:graphmc}, 
\[
	|A_N|^2 = \frac{H_N^2}{2} + |\mathring{A}_N|^2 =\frac{2}{R^2} + \frac{2}{R}\left(H_N - \frac{2}{R}\right) + O(R^{-2-2q}).
\]
By the definition of $L_N$, 
\begin{align} \label{eq:Ln}
	\int_N u L_N u \, d\sigma =& \int_n |\nabla^N u|^2 \,d\sigma- \frac{2}{R^2} - \int_N \frac{2}{R}\left(H_N - \frac{2}{R}\right) u^2\, d\sigma \notag \\
	&-\int_N Ric^M(\mu_g, \mu_g) u^2 \, d\sigma + O(R^{-2-2q}).
\end{align}
If we substitute the gradient term in the right-hand side by \eqref{eq:rough}, it eliminates the second term $2/R^2$. However, we do not know the sign of the remainders which are still of \emph{higher} order $O(R^{-2-q})$. Therefore, we have to derive a better estimate on the gradient term to cancel out the third term.  

Recall the Bochner--Lichnerowicz identity:
\begin{align*}
		\frac{1}{2} \Delta_{N} |\nabla^{N} u |^2 =& \left(\mbox{Hess}_{N} u \right)^2 + \langle \nabla^{N} u, \nabla^{N} \Delta_{N} u \rangle + \mathcal{K} (\nabla^{N} u, \nabla^{N} u)\\
		\ge & \frac{(\Delta_N u)^2}{2}+ \langle \nabla^{N} u, \nabla^{N} \Delta_{N} u \rangle + \mathcal{K} (\nabla^{N} u, \nabla^{N} u),
\end{align*}
where $\mathcal{K}$ is the Gauss curvature of $N$. After integrating the above inequality, the left hand side vanishes because $N$ is a compact manifold without boundary. Then using \eqref{eq:u} and \eqref{eq:E_2}, 
\[
	\int_N | \nabla^N u |^2 \, d\sigma \ge \frac{1}{R^2} + \frac{R^2}{2} \int_N \mathcal{K} | \nabla^N u|^2 \, d\sigma + O(R^{-2-2q}).
\]
Using the Gauss equation, \eqref{TraceFree} and \eqref{eq:graphmc},
\begin{align*}
		\mathcal{K} =&\frac{1}{2} (H_N^2 - |A|^2) -  Ric^M( \mu_g, \mu_g ) + \frac{1}{2} R_g \\
		=& \frac{1}{R^2} + \frac{1}{R} \left( H_N - \frac{2}{R} \right ) -  Ric^M(\mu_g , \mu_g ) + O(R^{-2-2q}),
\end{align*}
where we use that $R_g = O(R^{-2-2q})$ by the constraint equations \eqref{ConE}. Hence,
\begin{align*}
		\int_N | \nabla^N u |^2 \, d\sigma \ge& \frac{2}{R^2} + \int_N \left( H_N - \frac{2}{R}\right) \frac{R}{2} | \nabla^N u|^2 \, d\sigma\\
		& - \int_N Ric^M(\mu_g, \mu_g) \frac{R^2}{2} |\nabla^N u|^2 \, d\sigma + O(R^{-2-2q}).
\end{align*}
Substituting the above inequality back to \eqref{eq:Ln}, we have, for any $u$ satisfying that $\widetilde{u} \in \mathfrak{K}$ and $\| u \|_{L^2(N)} = 1$,
\begin{align*}
	\int_N u L_N u \, d\sigma \ge& \int_N \left( H_N - \frac{2}{R}\right) \left(\frac{R}{2} | \nabla^N u|^2 - \frac{2}{R} u^2 \right) \, d\sigma\\
	& -  \int_N Ric^M(\mu_g, \mu_g) \left(\frac{R^2}{2} |\nabla^N u|^2 + u^2 \right)\, d\sigma + O(R^{-2-2q}). 
\end{align*}
In particular, we choose $v_i$, for $i=1,2,3$, to satisfy
\[
	\widetilde{v_i} = \sqrt{\frac{3}{4\pi}} \frac{x^i - p^i }{ R^2}. 
\]
Then, for each $i$, because 
\[
	|\nabla^e \widetilde{v_i}|^2 =  \frac{3}{4\pi R^4 } - \frac{\widetilde{v_i}^2}{R^2},
\]
we get
\[
	| \nabla^N v_i|^2 = \frac{3}{4\pi R^4 } - \frac{v_i^2}{R^2} + O(R^{-4-q}).
\]
Hence, 
\begin{align*}
	\int_N v_i L_N v_i \, d\sigma \ge& \int_N \left( H_N - \frac{2}{R}\right) \left(\frac{3}{8\pi R^3} - \frac{3}{2R} v_i^2 \right) \, d\sigma\\
	& -  \int_N Ric^M(\mu_g, \mu_g) \left(\frac{3}{8\pi R^2}  + \frac{1}{2}v_i^2 \right)\, d\sigma + O(R^{-2-2q}). 
\end{align*}
Let $u = \sum_i v_i$. Then, because $\sum_i v_i^2 = 3/(4\pi R^2) + O(R^{-2-q})$,
\begin{align*}
	\int_N u L_N u \, d\sigma =& \sum_i \int_N v_i L_N v_i \, d\sigma + O(R^{-2-2q})\\
	\ge& -\frac{3}{4\pi R^2} \int_N  Ric^M(\mu_g, \mu_g)  \, d\sigma + O(R^{-2-2q}).
\end{align*}
We prove the claim.

To complete the proof, we use the alternative definition of the ADM mass (\ref{MASS2}) and obtain, 
\begin{align*}
	\int_N  Ric^M(\mu_g, \mu_g)  \, d\sigma &= \int_{S_R(p)} Ric^M (\nu_e, \nu_e) \, d\sigma_e + O(R^{-1-q}) \\
	&= -\frac{8\pi m}{R} + O(R^{-1-q}).
\end{align*}
\end{proof}

In order to apply the inverse function theorem, we prove that $L_N$ is invertible. We show that the lowest eigenvalue of $L_N$ \emph{without} any constraints is negative, and the next eigenvalue is strictly positive.
\begin{lemma} \label{INVERT}
Assume that $(M, g, K)$ is AF--RT at the decay rate $q\in (1/2,1]$ and $m>0$. Let $N$ be a normal graph of $\psi$ over $S_R(p)$:
\[	
		N = \left \{ \Psi(x)  = x + \psi \nu_g  : \psi \in C^2 \left(S_R (p) \right)\right\}, 
\]
where $\psi$ satisfies \eqref{eq:graphpsi} in Lemma \ref{lemma3}. For $R$ large, $L_N$ is invertible, and $L_N^{-1} : C^{0,\alpha} (N) \rightarrow C^{2,\alpha}(N) $ satisfies $ | L_N^{-1} | \le c m^{-1} R^3$.
\end{lemma}
\begin{proof}
Let $\eta_0$ be the lowest eigenvalue of $L_N$ without constraints. By Lemma \ref{lemma3}, 
\begin{align*} 
		\eta_0 =& \inf_{ \left\{ \| u \|_{L^2} = 1 \right\}} \int_N \left[ |\nabla^N u|^2 -\left ( |A_N|^2 + Ric^M( \mu_g, \mu_g) \right) u^2\right] \, d \sigma \notag \\
		\ge &- \frac{2}{R^2} + O(R^{-2-q}).
\end{align*}
On the other hand, if we replace $u$ by a constant, we obtain the reverse inequality. Hence,
\begin{eqnarray} \label{ETA}
		\eta_0 = -\frac{2}{R^2} + O(R^{-2-q}).
\end{eqnarray}
Let $h_0$ be the corresponding eigenfunction 
\[	
	L_N h_0 = \eta_0 h_0.
\] 
We show that $h_0$ is close to a constant and derive an $L^2$--estimate on the difference of $h_0$ and its mean value $\overline{h}_0 := |N |^{-1} \int_N h_0 \, d\sigma $.
\begin{align} \label{h_0}
		L_N (h_0 - \overline{h}_0) = \eta_0 (h_0 - \overline{h}_0 ) + \left(  \eta_0 + | A_N |^2 + Ric^M (\mu_g, \mu_g)\right) \overline{h}_0. 
\end{align}
Multiplying the above identity by $(h_0 - \overline{h}_0)$ and integrating it over $N$:
\begin{align*}
		\int_N | \nabla^N (h_0 - \overline{h}_0) |^2 \,d \sigma =& \int_N \left(\eta_0 + | A_N |^2 + Ric^M(\mu_g , \mu_g ) \right) (h_0 - \overline{h}_0)^2 \, d\sigma \\
		&+ \int_N (\eta_0 + | A_N |^2 + Ric^M(\mu_g , \mu_g ) ) (h_0 - \overline{h}_0) \overline{h}_0 \, d\sigma.
\end{align*}
Similarly as shown in the previous lemma, because $h_0 - \overline{h}_0$ has zero mean value, the left hand side is bounded below by 
\[
	\int_N | \nabla^N (h_0 - \overline{h}_0) |^2 \,d \sigma \ge \left( \frac{2}{R^2} + O\left(R^{-2-q}\right) \right) \int_N | h_0 - \overline{h}_0 |^2 \, d\sigma.
\]
Also, pointwisely
\begin{align} \label{eq:eta0}
	\eta_0 + |A_N |^2 + Ric^M(\mu_g , \mu_g )  = O(R^{-2-q}).
\end{align}
Therefore,
\begin{align*}
		&\frac{2}{R^2}  \int_N | h_0 - \overline{h}_0 |^2 \, d\sigma \le c R^{-2-q} \int_N |h_0 - \overline{h}_0 |^2 \, d\sigma + c R^{-2-q} \int_N |h_0 - \overline{h}_0| | \overline{h}_0| \, d\sigma.
\end{align*}  
Using the AM--GM inequality to the last integrand: 
\[
	 |h_0 - \overline{h}_0| | \overline{h}_0| \le \frac{1}{4c}R^q |h_0 - \overline{h}_0|^2 +  c R^{-q}| \overline{h}_0|^2, 
\] 
We obtain, when $R$ large,
\begin{align} \label{hmean}
		\|h_0 - \overline{h}_0 \|_{L^2(N)} \le c R^{-q} |\overline{h}_0| |N|^{1/2} .
\end{align}
In particular, $\overline{h}_0 \neq 0$. Let $\eta_1$ be the next eigenvalue with the corresponding eigenfunction $h_1$. We show that $\eta_1$ is positive and, moreover, 
\[
	\eta_1 \ge \frac{6m}{R^3} + O(R^{-2-2 q}).
\] 
Note that
$$
		0 = \int_N h_0 h_1 \, d\sigma = \int_N (h_0 - \overline{h}_0) (h_1 - \overline{h}_1) \, d\sigma + \int_N \overline{h}_0 h_1 \, d\sigma.
$$
Then, by H\"{o}lder's inequality, 
\[
		\left| \int_N h_1 \, d\sigma\right| \le | \overline{h}_0 |^{-1} \| h_0 - \overline{h}_0\|_{L^2(N)} \| h_1 - \overline{h}_1\|_{L^2(N)}.  
\]
Substituting \eqref{hmean} into the above inequality, we get
\begin{align}\label{H1}
		 \overline{h}_1 \le c R^{-q} |N|^{-1/2}\| h_1 - \overline{h}_1\|_{L^2(N)}.
\end{align}
Because $L_N h_1 = \eta_1 h_1$, 
\begin{align*}
		&\int_N (h_1 - \overline{h}_1) L_N (h_1 - \overline{h}_1)  \, d\sigma = \int_N\eta_1 (h_1 - \overline{h}_1)^2   \, d\sigma \\
		&\qquad \qquad + \int_N \overline{h}_1  (h_1 - \overline{h}_1) \left(\eta_1+|A_N|^2 + Ric^M (\mu_g, \mu_g) \right) \, d\sigma.
\end{align*}
Because $\eta_1 + |A_N|^2 + Ric^M(\mu_g, \mu_g)= \mbox{constant} + (R^{-2-q})$, and by H\"{o}lder's inequality, the last integral is bounded above by
\begin{align} \label{eq:H1}
		cR^{-2-q}|N|^{1/2} |\overline{h}_1| \| h_1 - \overline{h}_1 \|_{L^2(N)}.
\end{align}
By Lemma \ref{STABLE}, \eqref{H1}, and \eqref{eq:H1},
\begin{align*}
	\mu_0 \| h_1 - \overline{h}_1 \|^2_{L^2(N)} \le (\eta_1 + cR^{-2-2q} )\| h_1 - \overline{h}_1 \|^2_{L^2(N)} 
\end{align*}
Therefore, 
$$
		\eta_1 \ge \mu_0 + c R^{-2-2q}  \ge \frac{ 6m}{ R^3} + O\left(R^{-2-2q} \right).
$$
This finishes the proof.
\end{proof}


The family of constant mean curvature surfaces $\{\Sigma_R\}$ constructed in Theorem \ref{EXIST} satisfies the assumptions of $N$ in the previous two lemmas. They imply that, in particular, $\Sigma_R$ is strictly stable and $L_{\Sigma_R}$ is invertible. In the next theorem, we use the invertibility of $L_{\Sigma_R}$ and the inverse function theorem to show that $\{ \Sigma_R\}$ form a smooth foliation.

\begin{thm}\label{FOLI}
Assume that $(M,g,K)$ is AF--RT at the decay rate $q \in (1/2,1]$ and $m > 0$. Let $\{ \Sigma_R\}$ be the family of surfaces with constant mean curvature constructed in Theorem \ref{EXIST}. Then $\{ \Sigma_R \}$ form a smooth foliation in the exterior region of $M$.
\end{thm}
\begin{proof}
Let $\mathcal{H} : C^{2,\alpha} ( \Sigma_{R_1}) \rightarrow C^{0,\alpha} (\Sigma_{R_1})$ be the mean curvature map so that $\mathcal{H} (u)$ is the mean curvature of the normal graph of $u$ over $\Sigma_{R_1}$.

Because $d \mathcal{H} = - L_{ \Sigma_{R_1}}$ is a linear isomorphism by Lemma \ref{INVERT}, $\mathcal{H}$ is a diffeomorphism from a neighborhood $U$ of $0\in C^{2,\alpha}(\Sigma_{R_1})$ to a neighborhood $V$ of $\mathcal{H}(0)$ by the inverse function theorem. By our construction of $\{ \Sigma_R \}$, for $R$ close to $R_1$, $\{\Sigma_R\}$ are the unique constant mean curvature surfaces in a neighborhood of $\Sigma_R$. Moreover, $\{ \Sigma_R \}$ vary smoothly in $R$. 

To show that $\{ \Sigma_R \}$ form a foliation, we need to prove that $\Sigma_{R}$ and $ \Sigma_{R_1} $ have no intersection for any $R \neq R_1$. First, when $R$ is close to $R_1$ and $\Sigma_{R}$ is the graph of $u$ for $u\in U$, we show that $u$ has a sign; in particular, $u$ cannot be zero. In the following, we denote $\Sigma_{R_1}$ by $\Sigma$. 

By the Taylor theorem, for any $u\in U$, 
\begin{align*} 
		\mathcal{H}(u) = \mathcal{H}(0) - L_{\Sigma} u + \int_0^1  \left( d \mathcal{H}(su) - d \mathcal{H}(0)\right) u \, ds,
\end{align*}
where $\mathcal{H} (u) $ and $\mathcal{H} (0)$ are constants. By integrating the above identity over $\Sigma$, 
\begin{align} \label{eq:constantH}
	\mathcal{H}(0) - \mathcal{H}(u) = -\frac{2}{R^2} \overline{u} + E_8,
\end{align}
and 
\[
	E_8 \le cR^{-2-q} |u| + cR^{-3} \| u^* \|_{C^{2}(S_1(0))}. 
\]
We decompose $u = h_0 + u_0$ where $h_0$ is the lowest eigenfunction of $L_{\Sigma}$ and $\int_{\Sigma} h_0 u_0 \, d\sigma = 0$. Then 
\begin{align*}
		 | u - \overline{h}_0 | \le | h_0 - \overline{h}_0 | +  |u_0|.
\end{align*}
\noindent {\bf Claim}: The right hand side of the above inequality is small comparing to $\overline{h}_0$. Moreover precisely, 
\begin{align*}
		&\sup_{\Sigma} | h_0 - \overline{h}_0 |  \le c R^{-q} | \overline{h}_0|,\\
		&\sup_{\Sigma} |u_0| \le c R^{-q} | \overline{h}_0|.
\end{align*}

Assuming the claim, we obtain, by choosing $R$ large enough, 
\[
	 \overline{h}_0 - \frac{1}{2} \left|  \overline{h}_0 \right|  \le u \le \overline{h}_0 +  \frac{1}{2} \left|  \overline{h}_0 \right| .
\] 
Because $\overline{h}_0$ is nonzero by \eqref{hmean}, $u$ has a sign, and the theorem follows.   
\begin{cpf}
Recall that $h_0 - \overline{h}_0$ satisfies \eqref{h_0}. On a coordinate chart, \eqref{h_0} is a second order elliptic equation. We choose the coordinate chart to be a ball of radius $R$ on $\Sigma$. Then the number of charts to cover $\Sigma$ is independent of $R$. Using the De Giorgi--Nash--Moser theory\cite[Theorem 8.17]{GT01} on each chart and summing over the charts, we obtain 
\begin{align} \label{SUPh} 
		&\sup_{\Sigma}  | h_0 - \overline{h}_0 | \notag \\
		&\le c R^{-1} \| h_0 - \overline{h}_0 \|_{L^2(\Sigma)} + c  | \overline{h}_0 |  \left\| \eta_0 + |A_{\Sigma}|^2 + Ric^M(\mu_g, \mu_g) \right\|_{L^{2}(\Sigma)}\notag\\
		 &\le c R^{-q} | \overline{h}_0|, 
\end{align}
where we use \eqref{eq:eta0} and \eqref{hmean}. To prove the second identity in the claim, we need the H\"{o}lder estimate on $h_0 - \overline{h}_0$. By \cite[Theorem 8.22]{GT01} and \eqref{SUPh}.
\begin{align} \label{ALPHAh}
	&\left[ h_0 - \overline{h}_0 \right]_{0,{\alpha}} \notag \\
	&\le c \left( R^{-\alpha} \sup_{\Sigma} | h_0 - \overline{h}_0 | +  | \overline{h}_0 | \left\| \eta_0 + |A_{\Sigma}|^2 + Ric^M(\mu_g, \mu_g) \right\|_{L^{2}(\Sigma)} \right)  \notag \\
		&\le  c R^{-\alpha - q} | \overline{h}_0 |.    
\end{align}
To estimate $\sup_{\Sigma} |u_0|$, by the definition of $u_0$ and $h_0$,
\begin{align*}
		L_{\Sigma} u_0 = & L_{\Sigma} u -\eta_0 h_0 = - \frac{2}{R^2} \overline{u} - \eta_0 h_0 + E_8 + \int_0^1 (d \mathcal{H}(su) - d \mathcal{H}(0) u ) u \, ds\\
		 = & \frac{2}{R^2} (h_0 - \overline{h}_0) - \frac{2}{R^2} \overline{u}_0 + O(R^{-2-q}\| u^* \|_{C^{2}(S_1(0))}),
\end{align*}
where we use \eqref{eq:constantH} in the second equality and $\eta_0 =- 2/R^2 + O(R^{-2-q})$ in the third equality. Because $L_{\Sigma}$ has no kernel, by pulling back the equation to unit spheres and using the Schauder estimates,
\begin{align} \label{eq:u_0}
		\| u_0^* \|_{C^{2,\alpha}(S_1(0))} \le c \left( \| h_0^* - \overline{h}_0 \|_{C^{0, \alpha}(S_1(0))} +  | \overline{u}_0|+ R^{-q}\| u^* \|_{C^{2,\alpha}(S_1(0))}\right).
\end{align}
Because $h_0$ satisfies $L_N h_0= \eta_0 h_0$ and $\eta_0 = O(R^{-2})$, using the Schauder estimate on $h_0$ in the second inequality below, we have
\begin{align*}
	\| u^* \|_{C^{2,\alpha}(S_1(0))} &\le c \| u_0^* \|_{C^{2,\alpha}(S_1(0))}  + c \| h_0^* \|_{C^{2,\alpha}(S_1(0))}\\
	&\le c \| u_0^* \|_{C^{2,\alpha}(S_1(0))}  + c \| h_0^* \|_{C^{0,\alpha}(S_1(0))} \\
	& \le c \| u_0^* \|_{C^{2,\alpha}(S_1(0))}  + c \| h_0^* - \overline{h}_0 \|_{C^{0,\alpha}(S_1(0))}  + c | \overline{h}_0|.
\end{align*}
Therefore, combining the above identities and absorbing the term $c R^{-q}\| u^*_0 \|_{C^{2,\alpha}(S_1(0))}$ to the left of \eqref{eq:u_0} for $R$ large, we have
\begin{align} \label{u_0}
		\| u_0^* \|_{C^{2,\alpha}(S_1(0))} \le& c \left( \| h_0^* - \overline{h}_0 \|_{C^{0, \alpha}(S_1(0))}+ | \overline{u}_0| + R^{-q} | \overline{h}_0 | \right) \notag \\
		\le& cR^{-q} | \overline{h}_0| +  c  | \overline{u}_0|,
\end{align}
where we use \eqref{SUPh} and \eqref{ALPHAh} in the second inequality. 
It remains to estimate  $|\overline{u}_0|$. Because  $\int_{\Sigma} h_0 u_0 \, d\sigma = 0$, similarly as in (\ref{H1}), we have
\begin{align*}
		\left |\int_{\Sigma} u_0 \, d\sigma \right| \le 2  |N|  | \overline{h}_0 |^{-1} \sup_{\Sigma} | h_0 - \overline{h}_0 |   \sup_{\Sigma} | u_0 | 
\end{align*}
and then by (\ref{SUPh})
\begin{align*}
		| \overline{u}_0 | \le 2 | \overline{h}_0 |^{-1} \sup_{\Sigma}| h_0 - \overline{h}_0 |   \sup_{\Sigma} | u_0 | \le cR^{-q} \sup_{\Sigma} | u_0 |.
\end{align*}
Then $| \overline{u}_0 |$ could be absorbed into the left hand side of (\ref{u_0}) for $R$ large. 
\qed

We prove that in the neighborhood $U$ where the inverse function theorem holds, two surfaces in the family of $\{\Sigma_R\}$ have no intersection. Because the size of $U$ is independent of $R$ by the uniform bounds of $| d^2 \mathcal{H} | $ and $| L_{\Sigma}^{-1} |$ (c.f. \cite[Proposition 2.5.6]{AbrahamMarsdenRatiu88}), we could inductively proceed the argument toward infinity of $M$ and conclude that $\{ \Sigma_R \}$ form a foliation in the exterior region. 
\end{cpf}
\end{proof}


\section{Uniqueness of the Foliation}
In this section,  we assume that $(M,g, K)$ is AF--RT with $q \in ( 1/2, 1]$ and $m >  0$. $\Sigma_R$ is the surface with constant mean curvature constructed in Theorem \ref{EXIST}, and $\Sigma_R$ is a $c_0R^{1-q}$-graph over $S_R(\mathcal{C})$ as in Corollary \ref{cor:graph}.


\subsection{Local Uniqueness}

\begin{thm} \label{LU}
Assume that $N$ has constant mean curvature equal to $H_{\Sigma_R}$. Given any $c_1 \ge 2 c_0$, there exists $\sigma_1 = \sigma_1(c_1)$ so that, for $R \ge \sigma_1$, if $N$ is a $c_1 R^{1-q}$-graph over $S_R( \mathcal{C} )$, i.e.
\[
	N = \{ x + u \nu_g : u \in C^{2,\alpha}(S_R(\mathcal{C}))\}
\]
with 
\[
	\| u^* \|_{C^{2,\alpha}(S_1(0))} \le c_1 R^{1-q},
\]
then $N = \Sigma_R$.
\end{thm}
\begin{rmk}
Notice that we do not impose any condition on $(u^*)^{odd}$.
\end{rmk}
\begin{lemma} \label{lemma:localuniq}
There exists a constant $c_1^{\prime}$ so that $N$ is a $c_1^{\prime}$-graph over $\mathcal{S}(\mathcal{C}, R)$ and $N$ has constant mean curvature equal to $H_{\Sigma_R}$, then $N = \Sigma_R$.
\end{lemma}
\begin{pfl}
Assume that $N$ is the graph of $v$ over $\Sigma_R$. By using the invertibility of $L_{\Sigma_R}$, we first prove that there is a constant $c_1^{\prime}$ so that if  $\| v \|_{C^{2,\alpha} (\Sigma_R)} \le 2 c_1^{\prime}$, then $v \equiv 0$. 

By Taylor's theorem, and because $N$ and $\Sigma_R$ have the same mean curvature,
\begin{align*}
	 L_{\Sigma_R} v =\int_0^1 \left( d H_{\Sigma_R} (sv) - d H_{\Sigma_R} (0) \right) v \,d s.
\end{align*}
Because $ |L_{ \Sigma_R}^{-1} | \le cm^{-1}R^3$ by Lemma {\ref{INVERT}}, and by (\ref{ERRINT}), 
\begin{align*}
		\| v \|_{C^{2,\alpha} (\Sigma_R)} \le& cm^{-1} R^3\left \| \int_0^1 \left( d H_{\Sigma_R} (sv) - d H_{\Sigma_R} (0) \right) v  \,d \sigma \right \|_{C^{0,\alpha} (\Sigma_R)}\\
		 \le & cm^{-1} R^3 R^{-3} \| v\|_{C^{2,\alpha} (\Sigma_R)}^2 \le cm^{-1}\| v\|_{C^{2,\alpha} (\Sigma_R)}^2.
\end{align*}	
This implies $c^{-1} m \le \| v\|_{C^{2,\alpha} (\Sigma_R)}$. Choose any $c_1^{\prime} < (2c)^{-1}m$. If $\| v \|_{C^{2,\alpha} (\Sigma_R)} \le 2 c_1^{\prime}$,  then $v \equiv 0$.

By  the construction in Theorem {\ref{EXIST}} and \eqref{eq:psi}, $\Sigma_R$ is a $2^{-1} c_0R^{1-2q}$--graph over $\mathcal{S}(p ,R)$, and $p = \mathcal{C} + O(R^{1-2q})$ by Corollary \ref{GCTM}. For $R \ge \sigma_1= \sigma_1 (g, c_0, |p-\mathcal{C}|, c_1^{\prime})$ large, $\mathcal{S}(\mathcal{C}, R)$ is within $c_1^{\prime}$-distance of $\Sigma_R$. Also, because the normal vectors of $\Sigma_R$ and $\mathcal{S}(\mathcal{C}, R)$ are close for $R$ large, if $N$ is a $c_1^{\prime}$-graph over $\mathcal{S}(\mathcal{C}, R)$, then $N$ is a $2c_1^{\prime}$-graph over $\Sigma_R$. Therefore, by the above analysis, $N = \Sigma_R$.
\qed
\end{pfl}

\begin{pflu}
By the assumption, $N$ is the graph of $u$ over $S_R( \mathcal{C})$ with $\| u^*\|_{C^{2,\alpha}} \le c_1 R^{1-q}$. Because $p = \mathcal{C} + O(R^{1-2q})$, for $R$ large, we can assume that $N$ is the graph of $u$ over $S_R(p)$ with $\| u^*\|_{C^{2,\alpha}} \le 2 c_1 R^{1-q}$. Recall that $L_S$ denotes the linearized mean curvature operator on $S_R(p)$ with respect to $g$. By Taylor's theorem,
\begin{align} 
		H_S(u) =& H_S (0) - L_S u + \int_0^1 \left( d H_S (su) - d H_S (0) \right) u \,d s.
\end{align}
Also, recall that $L_0 = -\Delta_S^e - (2/R^2)$ and $\mathfrak{K} = \mbox{Ker}L_0$. Let $\phi$ be the function defined as in Lemma \ref{lemma:approx}; that is, $\mathcal{S}(p, R)$ is the graph of $\phi \in \mathfrak{K}^{\perp}$ over $S_R(p)$ and 
\[
	L_0 \phi = f - R^{-3-q} \sum_i A^i (x^i-p^i) - \overline{f}, 
\] 
where $f = H_S(0) - (2/R)$. Then we show that $u - \phi$ is small. Because $N$ and $\Sigma_R$ have the same mean curvature, $H_S(u) = 2/R + \overline{f}$ by the construction of $\Sigma_R$ in Theorem  \ref{EXIST}. Therefore,
\begin{align} \label{DIFF}
		L_0( u - \phi) = &R^{-3-q} \sum_i A^i (x^i - p^i)   +  ( L_0   - L_S) u\notag\\
		& + \int_0^1 \left( d H_S (su) - d H_S (0) \right) u \,d s.
\end{align}
We decompose $u$ into 
\begin{align*}
		u = u^{\perp} + R^{-q} \sum_i B^i (x^i - p^i),
\end{align*}
where $u^{\perp}\in \mathfrak{K}^{\perp}$ and,  for $i = 1,2,3$,
\begin{align*}
		B^i = \frac{3R^{-4+q}} {4\pi} \int_{S_R(p) } (x^i - p^i ) u \, d\sigma_e.
\end{align*}
Notice that we only use $|u| \le 2 c_1 R^{1-q}$ to guarantee $B^i=O(1)$, and we do not assume any condition on $u^{odd}$. Applying the Schauder estimates on to \eqref{DIFF}, because $u^{\perp} - \phi \in \mathfrak{K}^{\perp}$,
\begin{align} \label{eq:perpkernel}
		\| ( u^{\perp} - \phi)^* \|_{C^{2,\alpha}(S_1(0))} \le cR^{-q} \left( 1 + \| u^* \|_{C^{2,\alpha}(S_1(0))}\right) \le c(1+ 2c_1) R^{1-2q}.
\end{align}
To estimate the part inside the kernel, we first rewrite (\ref{DIFF}):
\begin{align} \label{eq:rewritten}
	&L_S \left[ R^{-q} \sum_i B^i (x^i - p^i) \right] = - L_0( u - \phi_0) + R^{-3-q} \sum_i A^i (x^i - p^i) \notag  \\
	&\qquad \qquad + ( L_0   - L_S) u^{\perp}+ \int_0^1 \left( d H_S (su) - d H_S (0) \right) u \,d s.
\end{align}
Then we multiply the above identity by  $\sum_i B^i (x^i - p^i)$ and integrate it over $S_R(p)$ with respect to the area measure $d\sigma$. First notice that, by \eqref{APPROX}, \eqref{eq:A}, and Lemma \ref{lemma1}, for $a = 1,2,3$,
\begin{align*}
		 \int_{S_R( p)}  (x^a- p^a) R^{-3-q} \sum_i A^i (x^i - p^i) \, d \sigma_e  &= 8\pi m (p^a- \mathcal{C}^a) + O(R^{-q}) \\
		 &=  O(R^{1-2q}).
\end{align*}
Also,
\begin{align*}
	&\int_{S_R(p)}  (x^a - p^a )( L_0   - L_S) u^{\perp} \, d\sigma \\
	&= \int_{S_R(p)}  (x^a - p^a )( L_0   - L_S) (u^{\perp} - \phi)\, d\sigma +  \int_{S_R(p)}  (x^a - p^a )( L_0   - L_S)  \phi \, d\sigma. 
\end{align*}
Combining \eqref{eq:perpkernel} and the fact that $\| (\phi^*)^{odd} \|_{C^{2,\alpha}} \le cR^{-q}$, the above term is $O(R^{1-2q})$. For other terms in the right hand side of \eqref{eq:rewritten}, they are of order $O(R^{1-2q})$ after integrating with $\sum_i B^i (x^i - p^i)$ as well. Concluding the above estimates and using the eigenvalue estimate on $\mu_0$ in Lemma \ref{STABLE} (for the operator $L_S$ on $S_R(p)$), 
\begin{align*} 
		&\left(\frac{6m}{R^3} + O(R^{-3-q})\right) R^{-q} \| \sum_i B^i (x^i - p^i) \|_{L^2(S_R(p))} \le c R^{1-2q}.
\end{align*}
That is, the bound on $B^i, i=1,2,3,$ is improved:
\begin{align} \label{eq:projection}
	|B^i| \le cR^{-q}.
\end{align}
Therefore, using \eqref{eq:perpkernel} and \eqref{eq:projection},   
\begin{align*}
		\| (u - \phi)^* \|_{C^{2,\alpha}} \le& \| (u^{\perp} - \phi)^* \|_{C^{2,\alpha}} + \left \| R^{1-q} \sum_i B^i y^i  \right\|_{C^{2,\alpha}}\le 2 c(1+c_1) R^{1-2q}.
\end{align*}
By choosing $R \ge \sigma_1 = \sigma_1(\mu_0, \| \phi^* \|_{C^{2,\alpha}}, \| (\phi^*)^{odd} \|_{C^{2,\alpha}}, c_1)$, we have
\begin{align*}
		\| (u - \phi)^* \|_{C^{2,\alpha}}  \le \frac{c_1^{\prime}}{2}.
\end{align*}
Because the normal vectors of $S_R(p)$ and of $\mathcal{S}(\mathcal{C}, R)$ are close enough, we could arrange $N$ to be a $c_1^{\prime}$--graph over $\mathcal{S}(\mathcal{C}, R)$. Then by Lemma \ref{lemma:localuniq},  $N= \Sigma_R$.
\qed
\end{pflu}

The above theorem says that, among surfaces which are spherical and close to  the Euclidean sphere centered at $\mathcal{C}$,  $\Sigma_R$ is the only one with the constant mean curvature $H_{\Sigma_R}$. In particular, we can generalize the above results to the spherical constant mean curvature surfaces. 
\begin{cor} \label{LU2}
Assume $| p -\mathcal{C} | \le c_3 R^{1-q}$. Given any $c_4 \ge 2(c_0 + c_3)$,  there exists $\sigma_1 = \sigma_1(c_0, c_3, c_4)$ so that, for $R \ge \sigma_1$, if $N$ has constant mean curvature equal to $H_{\Sigma_R}$, and if $N$ is a $c_4 R^{1-q}$-graph over $S_R( p )$, then $N = \Sigma_R$.
\end{cor}
\begin{proof}
Assume that $N$ is a $c_4 R^{1-q}$-graph over $S_R(p)$. Because the normal vectors on $S_R(p)$ and $S_R(\mathcal{C})$ are close and $| p - \mathcal{C} | \le c_3 R^{1-q}$, $N$ is a $(c_0 +c_3) R^{1-q}$-graph over $S_R( \mathcal{C})$ for $R$ large. Then we can apply Theorem  \ref{LU} (by letting $c_1 = c_0 + c_3$) and derive that $N = \Sigma_R$.
\end{proof}

\subsection{A Priori Estimates}

In this subsection, we assume $(M,g,K)$ is AF at the decay rate $q\in(1/2,1]$ (note that the RT condition is not  assumed). For general surfaces $N$ in $M$ with constant mean curvature, we would like to derive a priori estimates and show that they are spherical under the condition that $N$ is stable.

Let $N$ be a smooth surface with constant mean curvature $H$ and $N$ be topologically a sphere. Assume that $N$ is stable, i.e.
\[
	\int_N u L_N u \, d\sigma \ge 0, \quad \mbox{for all  $u$ satisfying $\int_N u \, d\sigma =0$}.
\]
Let the minimum radius and the maximal radius of $N$ be defined by $\underline{r} = \min \{ |z| : z \in N \}$ and $\overline{r} =\max \{ |z| : z \in N\}$ respectively. $A$ denotes the second fundamental form of $N$, and $\mathring{A} = A - \frac{1}{2} H g_{N}$ denotes the trace-free part of $A$. $\mu_g$ is the outward unit normal vector field on $N$, and $\Delta$ and $\nabla$ are the Laplacian and the covariant derivative on $N$ with respect to the induced metric $g_N$. Moreover, we denote $R_{ijkl}$ or Riem the Riemannian curvature tensor and $Ric$ the Ricci curvature tensor of $(M, g, K)$ respectively. 

The following Sobolev inequality can be found in, for example, \cite[Proposition 5.4]{HY96}.
\begin{SobIneq}
For $\underline{r}$ large, there is a constant $c_{sob}$ so that for any Lipschitz functions $v$ on $N$, 
\begin{eqnarray} \label{SOBOLEV}
		\left( \int_N v^2 \, d\sigma \right)^{\frac{1}{2} } \le c_{sob} \int_N (| \nabla v | + H |v| ) \, d\sigma
\end{eqnarray}
\end{SobIneq}
\begin{lemma} \label{BASIC}
Assume that $N$ is a smooth surface in $M$ with constant mean curvature $H$. Also, assume that $N$ is topologically a sphere and stable. Then there is some constant $c$ so that the following estimates hold for $\underline{r}$ large, 
\begin{enumerate}
\item 
For any $ s > 2$, $ \displaystyle		\int_N | x|^{-s} \, d\sigma \le c \underline{r}^{2-s}$,
\item 
$ \displaystyle
		\left \| | \mathring{A} | \right\|_{L^2} \le c \underline{r}^{-\frac{q}{2}},
$
\item
$ \displaystyle
		c_{sob}^{-1} \le H^2 | N| \le c.
$
\end{enumerate}
\end{lemma}
\begin{proof}
Using the first variation formula as in \cite[Lemma 5.2]{HY96}, for any $s> 2$,
\begin{align*}
		\int_N |x|^{-s} \, d\sigma \le c \underline{r}^{2-s} H^2 | N |.
\end{align*}
Because $N$ is topologically a sphere, by the stability condition as in \cite[Proposition 5.3]{HY96} and the fact that the Ricci curvature is bounded by $|x|^{-2-q}$, we have
\begin{align*}
		\int_N | \mathring{A}|^2 \, d \sigma  \le c \underline{r}^{-q} H^2 | N |.
\end{align*}
If $(3)$ holds, especially the upper bound, then both $(1)$ and $(2)$ directly follow.

The lower bound in $(3)$ can be derived by letting $|v| = H$ in the Sobolev inequality (\ref{SOBOLEV}). Let $\mathcal{K}$ be the Gauss curvature of $N$. For the upper bound, the Gauss equation imply
\begin{align*}
		\int_N \frac{1}{2} H^2 \, d\sigma = & \int_N \left[ 2\mathcal{K} + |\mathring{A}|^2 - R_g + 2 Ric(\mu_g, \mu_g) \right] \, d\sigma \\
		\le& c + c    \int_N ( |\mathring{A}|^2 + |x|^{-2-q}) \, d\sigma \\
		\le& c + c \underline{r}^{-q} H^2 |N|.
\end{align*}
 For $\underline{r}$ large, the last term is absorbed to the left hand side, and $(3)$ is proved.
\end{proof}


Assume that the Greek letters range over $\{ 1, 2\}$, and the Latin letters range over $\{ 1,2,3\}$. For any surface $N$ in $M$,  the Simons identity \cite{Simons68} states
\begin{align*}
		\Delta A_{\alpha \beta } =& \nabla_{\alpha} \nabla_{\beta} H + H A^{\delta}_{\alpha} A_{\delta \beta} - | A |^2 A_{\alpha \beta} + A^{\delta}_{\alpha} R_{\epsilon \beta \epsilon \delta} + A^{\delta \epsilon} R_{\delta \alpha \beta \epsilon} \\
		& + \nabla_{\beta} \left( Ric_{\alpha k} \nu^k \right) + \nabla^{\delta} \left( R_{k \alpha \beta \delta } \nu^k \right).
\end{align*}
Because $H$ is a constant, the Simons identity gives an equation on $\mathring{A}$. We show that $\mathring{A}$ is small in the following lemma. 
\begin{lemma} \label{L2}
\begin{align*}
		\left\| | \mathring{A} |^2 \right\|_{L^2} +\left \| \nabla | \mathring{A} | \right\|_{L^2} + \left\| | \nabla \mathring{A} | \right\|_{L^2} + \left\|  H | \mathring{A} |\right\|_{L^2} \le c \underline{r}^{-1-q}.
\end{align*}
\end{lemma}
\begin{proof}
First by the Cauchy--Schwarz inequality $| \nabla \mathring{A} |^2 \ge \left|  \nabla | \mathring{A} | \right|^2$. Then by direct computations and the Codazzi equation (see \cite[Corollary 3.5]{Metzger07} or \cite[p. 237]{SY81b}):
\begin{align*}
		| \nabla \mathring{A} |^2 - \left| \nabla | \mathring{A} | \right|^2 \ge& \frac{1}{17} | \nabla \mathring{A} |^2 - \frac{16}{17} \left( | \omega |^2 + | \nabla H |^2\right)\\
		\ge&	\frac{1}{34}| \nabla \mathring{A} |^2 + \frac{1}{34} \left| \nabla | \mathring{A} | \right|^2  - \frac{16}{17} \left( | \omega |^2 + | \nabla H |^2\right),
\end{align*}
where $\omega = Ric(\cdot, \mu_g)^T$ denotes the projection of $Ric(\cdot, \mu_g )$ onto the tangent space of $N$. Substitute the above inequality into the following identity:
\begin{align} \label{eq:traceA}
		2 | \mathring{A} | \Delta | \mathring{A} | + 2 \left| \nabla | \mathring{A} | \right|^2 = \Delta | \mathring{A}|^2 = 2 \mathring{A}^{\alpha \beta} \Delta \mathring{A}_{\alpha \beta} +2 | \nabla \mathring{A} |^2.
\end{align}
Then we have
\begin{align*}
		 | \mathring{A} | \Delta | \mathring{A} | \ge \mathring{A}^{\alpha \beta} \Delta \mathring{A}_{\alpha \beta}+ \frac{1}{34}| \nabla \mathring{A} |^2 + \frac{1}{34} \left| \nabla | \mathring{A} | \right|^2  - \frac{16}{17} \left( | \omega |^2 + | \nabla H |^2\right).
\end{align*}

Because $H$ is a constant, we use the Simons identity in the above inequality and have
\begin{align} \label{TRACE}
		 | \mathring{A} | \Delta | \mathring{A} | \ge& H \mathring{A}^{\alpha \beta} A^{\delta}_{\alpha} A_{\delta \beta} - |A|^2 \mathring{A}^{\alpha \beta} A_{\alpha \beta} +  \mathring{A}^{\alpha \beta} A^{\delta}_{\alpha} R_{\epsilon \beta \epsilon \delta} + \mathring{A}^{\alpha \beta} A^{\delta \epsilon} R_{\delta \alpha \beta \epsilon}\notag \\
		 &
		 + \mathring{A}^{\alpha \beta} \nabla_{\beta} \left( Ric_{\alpha k} \nu^k\right) + \mathring{A}^{\alpha \beta} \nabla^{\delta} \left( R_{k \alpha \beta \delta} \nu^k\right)\notag \\
		 & + \frac{1}{34}| \nabla \mathring{A} |^2 + \frac{1}{34} \left| \nabla | \mathring{A} | \right|^2  - \frac{16}{17}  | \omega |^2.
\end{align} 
A direct calculation shows that the first two terms on the right hand side is
\begin{align*}
		  H \mathring{A}^{\alpha \beta} A^{\delta}_{\alpha} A_{\delta \beta}  - |A|^2 \mathring{A}^{\alpha \beta} A_{\alpha \beta}  =- ( | A|^2 - H^2 ) |\mathring{A} |^2 + H \mathring{A}^{\alpha \beta} \mathring{A}^{\delta}_{\alpha} \mathring{A}_{\delta \beta}. 
\end{align*}
The last term is the sum of cubic of the eigenvalues of $\mathring{A}$, which vanishes because $\mathring{A}$ is trace-free and $N$ is two-dimensional. Then integrating $ - | \mathring{A} | \Delta | \mathring{A} | $ over $N$ yields
\begin{align} \label{TRACEFREE}
		&\int_N \left( \frac{35}{34}  | \nabla | \mathring{A} | |^2 + \frac{1}{34}| \nabla \mathring{A} |^2 \right) \, d\sigma \notag\\
		&\le \int_N  ( | A|^2 - H^2 ) |\mathring{A} |^2 \, d\sigma -  \int_N\left( \mathring{A}^{\alpha \beta} A^{\delta}_{\alpha} R_{\epsilon \beta \epsilon \delta}  +  \mathring{A}^{\alpha \beta} A^{\delta \epsilon} R_{\delta \alpha \beta \epsilon}  \right)\, d \sigma \notag \\
		& \quad- \int_N \left[  \mathring{A}^{\alpha \beta} \nabla_{\beta} \left(  Ric_{\alpha k} \nu^k \right)  + \mathring{A}^{\alpha \beta} \nabla^{\delta} \left( R_{k \alpha \beta \delta} \nu^k\right) \right]\, d\sigma +\int_N   \frac{16}{17}  | \omega |^2 \, d \sigma.
\end{align}
The last term in the second line can be bounded by 
\begin{align*}
		c \int_N \left( | \mathring{A} |^2 | \mbox{Riem} | + H | \mathring{A} | | \mbox{Riem} | \right) \, d\sigma \le c \int_N |x|^{-2-q} (|\mathring{A}|^2 + H | \mathring{A}|) \, d\sigma.
\end{align*}
Using integration by parts and the Codazzi equation, the first integral in the third line can be bounded by $c \int_N  | \omega |^2 \, d\sigma$. To estimate the first integral in the second line of \eqref{TRACEFREE}, we use the stability condition. Because $N$ is stable, for any $u$ with mean value $\overline{u}$, by the stability equation for $u - \overline{u}$:
\begin{align*}
		&\int_N | A |^2 u^2 \, d\sigma \\
		&\le  \int_N | \nabla u |^2 \, d \sigma + \int_N | A |^2 (2 u \overline{u} - \overline{u}^2) \, d\sigma - \int_N Ric(\mu_g, \mu_g) (u - \overline{u})^2 \, d\sigma\\
	&\le \int_N | \nabla u |^2 \, d \sigma + \int_N \left( | \mathring{A} |^2 + \frac{1}{2} H^2 \right) (2 u \overline{u} - \overline{u}^2) \, d\sigma  + 2 \int_N | Ric(x) | (u^2 + \overline{u}^2) \, d\sigma.
\end{align*}
Because $ 2u \overline{u} - \overline{u}^2 \le u^2$ and $| Ric(x) | \le c |x|^{-2-q}$, we let $u = | \mathring{A} |$ and rewrite the above inequality as follows:
\begin{align*}
		\int_N \left( | A|^2 - \frac{1}{2} H^2 \right) | \mathring{A} |^2 \, d\sigma \le& \int_N | \nabla | \mathring{A}| |^2 \, d \sigma + 2 \overline{u} \int_N | \mathring{A}|^3 \, d\sigma \\
		&+ 2 c\int_N |x|^{-2-q} \left( | \mathring{A}|^2 + \overline{u}^2 \right) \, d\sigma.
\end{align*}
Multiplying the above inequality by $69/68$, and adding it to (\ref{TRACEFREE}),
\begin{align*}
		 &\int_N | \mathring{A} |^4 \, d\sigma + \int_N \left| \nabla | \mathring{A} | \right|^2 \, d\sigma +  \int_N | \nabla \mathring{A} |^2 \, d\sigma + H^2 \int_N | \mathring{A} |^2 \, d\sigma\\
		 &\le c \overline{u} \int_N | \mathring{A} |^3 \, d\sigma + c \int_N |x|^{-2-q} \left( | \mathring{A} |^2 + \overline{u}^2 \right) \, d\sigma + c \int_NH | \mathring{A} |   |x|^{-2-q} \, d\sigma \\
		 &\quad + c \int_N |x|^{-4-2q} \, d\sigma.
\end{align*}
Because $\| \mathring{A} \|_{L^2} \le c \underline{r}^{-\frac{q}{2}}$ by Lemma \ref{BASIC} (2), by the H\"older inequality and Lemma \ref{BASIC} (3), 
\begin{align*}
		\overline{u}^2 := |N|^{-2} \left( \int_N | \mathring{A} | \, d\sigma\right)^2 \le |N|^{-1} \int_N |\mathring{A}|^2 \, d\sigma \le c |N|^{-1} \underline{r}^{-q} \le c \underline{r}^{-q} H^2.
\end{align*}
By the AM--GM inequality and the above identity, 
\begin{align*}
		c \overline{u} \int_N | \mathring{A} |^3 \, d\sigma \le \frac{1}{4} \int_N | \mathring{A} |^4 \, d\sigma + c\underline{r}^{-q} H^2 \int_N | \mathring{A} |^2 \, d\sigma.
\end{align*}
For $\underline{r}$ large enough, these two terms could be absorbed to the left hand side. Similarly, we estimate the rest of terms 
\begin{align*}
	  &c \int_N |x|^{-2-q} \left( | \mathring{A} |^2 + \overline{u}^2 \right) \, d\sigma + c \int_N |x|^{-4-2q} \, d\sigma \\
	 &  \le \frac{1}{4} \int_N | \mathring{A} |^4 \, d\sigma + \underline{r}^{-2-q} \overline{u}^2 | N|   +c \underline{r}^{-2-2q},
\end{align*}
and
\begin{align*}
	  c \int_N H | \mathring{A} |  |x|^{-2-q} \, d\sigma \le \frac{1}{2} H^2 \int_N  | \mathring{A}|^2\, d\sigma + c \underline{r}^{-2-2q}.
\end{align*}
We then derive 
\begin{align*}
		\left\| | \mathring{A} |^2 \right\|_{L^2} +\left \| \nabla | \mathring{A} | \right\|_{L^2} + \left\| | \nabla \mathring{A} | \right\|_{L^2} + \left\|  H | \mathring{A} |\right\|_{L^2} \le c \underline{r}^{-1-q}.
\end{align*}
\end{proof}
\begin{rmk}
In particular, comparing to Lemma \ref{BASIC} (2),  the $L^2$ bound of $|\mathring{A}|$ is improved:
\begin{align} \label{ineq:traceA}
	\left\| | \mathring{A}| \right\|_{L^2} \le c H^{-1} \underline{r}^{-1-q}.
\end{align}
\end{rmk}

\subsection{The Position Estimate}
In this subsection, we assume that $(M,g,K)$ is AF at the decay rate $q\in (1/2,1]$ (note that the RT condition is not  assumed). Assume that $N$ has constant mean curvature $H$ and that $N$ is stable. In order to prove that $N$ is spherical, we derive the pointwise estimate of $| \mathring{A}|$ by the  $L^p$-estimates on $|\mathring{A}|$ in the previous subsection and the Simons identity for $|\mathring{A}|$. Inspired by \cite{QT07}, we apply the Moser iteration to functions satisfying this type of the differential equation below. 
\begin{lemma} \label{POINT}
For any functions $u \ge 0$, $f \ge 0$, and $h$ on $N$ satisfying
\begin{align} \label{GE}
				- \Delta u \le f u + h,
\end{align}
we have the pointwise control on $u$ as follows:
\begin{align*}
		\sup_N u \le c (\| f \|_{L^2} + H + \underline{r}^{-1} ) ( \| u \|_{L^2} + \underline{r} H^{-1} \| h \|_{L^2} ).
\end{align*}
\end{lemma}
\begin{proof}
Replacing $v$ by $v^2$ in the Sobolev inequality (\ref{SOBOLEV}) and using the H\"{o}lder inequality, we derive a variant of the Sobolev inequality
\begin{align} \label{SOB}
	\left(\int_N  v^4  \, d \sigma \right)^{ \frac{1}{2} } \le& c \left(\int_N |v| | \nabla v|  \, d \sigma + \int_N H v^2  \, d \sigma \right) \notag \\
		\le & c \left(  \int_N v^2  \, d \sigma \right)^{\frac{1}{2}}  \left[ \left( \int_N| \nabla v |^2  \, d \sigma \right)^{\frac{1}{2}}  + \left(\int_N H^2 v^2  \, d \sigma \right)^{\frac{1}{2}} \right].
\end{align}
Let $k$ be a positive constant and $\hat{u} =u + k$. Then multiplying $\hat{u}^{p-1}$ on the both sides of  (\ref{GE}),
\begin{align} \label{eq:tildeu}
		- \hat{u}^{p-1} \Delta \hat{u} \le& f \hat{u}^{p} - k f \hat{u}^{p-1} +  \frac{h}{\hat{u}} \hat{u}^p\le f \hat{u}^p + \frac{h}{k} \hat{u}^p =  \hat{f} \hat{u}^p,
\end{align}
where $\hat{f} = f + k^{-1} h$. Integrating \eqref{eq:tildeu} and using
\[
	| \nabla (\hat{u}^{\frac{p}{2}}) |^2 = \frac{p^2}{4 (p-1)} (p-1) \hat{u}^{p-2}| \nabla \hat{u} |^2,
\]
we have, for $p\ge 2$,
\begin{align*}
		&\int_N \left | \nabla ( \hat{u}^{\frac{p}{2} }) \right|^2 \, d \sigma  = \frac{p^2}{4 (p-1) } \int_N - \hat{u}^{p-1} \Delta \hat{u} \, d\sigma \le p \int_N  \hat{f} \hat{u}^p  \, d\sigma.
\end{align*}
We let $v$ be $\hat{u}^{\frac{p}{2}}$ in (\ref{SOB}) and substitute the gradient term by the above inequality. Then, 
\begin{align*}
		\left( \int_N  \hat{u}^{2p} \, d\sigma \right)^{\frac{1}{2} } \le c \left(\int_N \hat{u}^p  \, d\sigma \right)^{\frac{1}{2}}\left[ \left( p \int_N \hat{u}^p \hat{f} \, d\sigma \right)^{\frac{1}{2} }+ \left( \int_N  H^2 \hat{u}^p \, d\sigma \right)^{\frac{1}{2} } \right].
\end{align*}
By the H\"{o}lder inequality, the last two terms can be bounded by 
\begin{align*}
		& \left( p \int_N \hat{u}^p \hat{f} \, d\sigma \right)^{\frac{1}{2} } \le p^{\frac{1}{2}} \left( \int_N \hat{u}^{2p} \, d\sigma \right)^{\frac{1}{4}} \left(\int_N \hat{f}^2  \, d\sigma\right)^{\frac{1}{4}},\\
		 & \left( \int_N  H^2 \hat{u}^p \, d\sigma \right)^{\frac{1}{2} }  \le \left( \int_N H^4  \, d\sigma \right)^{\frac{1}{4}} \left( \int_N \hat{u}^{2p} \, d\sigma \right)^{\frac{1}{4}}.
\end{align*}
Therefore, using the above inequalities and the AM--GM inequality, 
\begin{align*}
		\left( \int_N  \hat{u}^{2p} \, d\sigma \right)^{\frac{1}{2} } \le& \frac{1}{2}  \left( \int_N \hat{u}^{2p} \, d\sigma \right)^{\frac{1}{2}} + c p \left( \int_N \hat{u}^p \, d\sigma \right)\left(\int_N \hat{f}^2 \, d\sigma\right)^{\frac{1}{2} } \\
		& + c \left( \int_N \hat{u}^p \, d\sigma \right) \left(\int_N H^4 \, d\sigma\right)^{\frac{1}{2} }. 
\end{align*}
Therefore, 
\begin{align*}
		 \left( \int_N  \hat{u}^{2p} \, d\sigma \right)^{\frac{1}{2} }  \le c p \left[ \left ( \int_N \hat{f}^2 \, d\sigma \right)^{\frac{1}{2} } + \left( \int_N H^4 \, d\sigma\right)^{\frac{1}{2}}\right] \left( \int_N \hat{u}^p \, d\sigma \right).
\end{align*}
Then, using Lemma \ref{BASIC} (3) to bound $H^2 |N|^{1/2} \le c H$, we obtain
\begin{align*}
		\left( \int_N \hat{u}^{2p} \, d\sigma \right)^{\frac{1}{2p}} \le c^{\frac{1}{p}} p^{\frac{1}{p} }\left( \| \hat{f} \|_{L^2} + H \right)^{\frac{1}{p} } \left( \int_N \hat{u}^p \, d\sigma \right)^{\frac{1}{p}}.
\end{align*}
Now letting $p=2^i, i = 1,2,3,\dots$, we then have
\begin{align*}
		\left( \int_N \hat{u}^{2^{l+1}} \, d\sigma \right)^{2^{-l-1}} \le\left[ c\left(  \| \hat{f} \|_{L^2} + H \right)\right]^{\sum_{i=1}^l 2^{-i}} 2^{\sum_{i=1}^l (i2^{-i})} \| \hat{u} \|_{L^2}.
\end{align*}
Let $l \rightarrow \infty$,
\begin{align*}
		\sup_N \hat{u} \le& c\left(  \| \hat{f} \|_{L^2} + H \right) \| \hat{u} \|_{L^2} \\
		\le& c \left( \| f \|_{L^2} + k^{-1} \| h \|_{L^2} + H \right) \left( \left\| u \right \|_{L^2} + k H^{-1}\right),
\end{align*}
where we use $|N|^{1/2}  \le c H^{-1}$. Let $k = \underline{r} \| h \|_{L^2}$. Then the proof is completed. 
\end{proof}


\begin{cor} \label{A}
\begin{align*}
		\sup | \mathring{A} | \le c (\underline{r}^{-1-q} + H^{-1} \underline{r}^{-2-q}). 
\end{align*}
Furthermore, if $\underline{r} \ge H^{-a}$ for some fixed $a \le 1$, then
\begin{align*}
		\sup | \mathring{A} | \le c H^{1+ \epsilon},
\end{align*}
where $\epsilon = (2 + q ) a - 2$, and $\epsilon > 0$ if $\displaystyle\frac{2}{2+q} < a \le 1$.  
\end{cor}
\begin{proof}
Notice that by \eqref{eq:traceA} and the Cauchy--Schwarz inequality, 
\[
	- | \mathring{A}| \Delta |\mathring{A} | \le - \mathring{A}^{\alpha \beta} \Delta \mathring{A}_{\alpha \beta}.
\] 
Using the Simons identity and the estimates in Lemma \ref{L2}, we have 
\begin{align*} 
		-  | \mathring{A} | \Delta | \mathring{A} | \le& ( |A|^2 - H^2 ) | \mathring{A} |^2 -  \mathring{A}^{\alpha \beta} A^{\delta}_{\alpha} R_{\epsilon \beta \epsilon \delta} - \mathring{A}^{\alpha \beta} A^{\delta \epsilon} R_{\delta \alpha \beta \epsilon}\notag \\
		 & - \mathring{A}^{\alpha \beta} \nabla_{\beta} \left( Ric_{\alpha k} \nu^k\right) - \mathring{A}^{\alpha \beta} \nabla^{\delta} \left( R_{k \alpha \beta \delta} \nu^k\right) \\
		 \le& c \left( | \mathring{A} |^4 + | \mathring{A} |^2 |x|^{-2-q} +H | \mathring{A} |  |x|^{-2-q} + | \mathring{A} | |x|^{-3-q} \right), 
\end{align*}
where we have used that $| R_{ijkl} | \le c |x|^{-2-q}$ and $| \nabla R_{ijkl} | \le c |x|^{-3-q}$. Set
\begin{align*}
		u &= |\mathring{A}|,\\
		f &= c ( | \mathring{A} |^2 + |x|^{-2-q}),\\
		h &= c ( H |x|^{-2-q} + |x|^{-3-q}).
\end{align*}
By Lemma \ref{L2}, 
\begin{align*}
		\| u \|_{L^2} \le c H^{-1} r^{-1-q}, \qquad \| f \|_{L^2} \le c \underline{r}^{-1-q},  \qquad \| h \|_{ L^2} \le c \underline{r}^{-2-q}.
\end{align*}
The corollary follows by Lemma \ref{POINT}.
\end{proof}

Because $M$ is AF, the estimates on $| \mathring{A} |$ yields the estimates on $|\mathring{A}^e|$ when $N$ is treated as an embedded surfaces in Euclidean space. We prove that $N$ is a graph over the sphere $S_{r_0}(p)$.  

The following lemma is a generalization of \cite[Proposition 2.1]{HY96} where that $M$ was assumed strongly asymptotically flat. A similar argument allows us to generalize to AF manifolds at the decay rate $q>1/2$ and to remove the conditions on $|\nabla \mathring{A}|$ and $\overline{r}$. We include the proof for completeness. 

\begin{lemma} \label{POSITION}
Let $N$ satisfy the assumptions as in Theorem \ref{GU1}. Then, there exists the center $p$ so that for all $z \in N$,
\begin{eqnarray}
		&& | {\lambda^e}_i  - r_0^{-1} | \le c H^{1+\epsilon} \label{RADIUS}\\
		&&\left |  \nu_e(z) - \frac{z-p}{r_0} \right| \le c H^{\epsilon} \label{GRAPH}
\end{eqnarray}
where  $r_0 = 2/H$, ${\lambda^e}_i$ and $\nu_e(z)$ are the principal curvature and the outward unit normal vector at $z$ with respect to the Euclidean metric. Moreover, $N$ is a graph over $S_{r_0}(p)$ so that  
\begin{align*}
		N =\left \{ z = x + \nu_g v : x \in S_{r_0}(p), v \in C^{1}(S_{r_0}(p))  \right\}
\end{align*}
and $\| v^* \|_{C^{1}} \le c H^{-1+\epsilon}$.
\end{lemma}

\begin{proof}
By Corollary {\ref{A}}, $\sup_N |\mathring{A}| \le c H^{1+ \epsilon}.$ Because $M$ is AF and $\underline{r} \ge H^{-a}$, for $\underline{r}$ large,
\begin{align*}
		&\sup_N |\mathring{A}^e| \le \sup_N|\mathring{A}| + c  \underline{r}^{-1-q} \le c H^{1+ \epsilon},\\
		& | H^e - H | \le c \underline{r}^{-1-q} \le c H^{1+\epsilon}.
\end{align*}
We would like to use the bound of these Euclidean quantities to show that $N$ is close to some sphere in the Euclidean space. To derive (\ref{RADIUS}), 
\begin{align*}
		\left| {\lambda^e}_i  - \frac{1}{2} H \right| \le& \left|  {\lambda^e}_i - \frac{1}{2}H^e \right| +  \left|  \frac{1}{2}H^e - \frac{1}{2}H \right| \\
		\le&		| \mathring{A}^e | + c H^{1+\epsilon}\le c H^{1+\epsilon}.
\end{align*}
Let $r_0^{-1} = (1/2) H $, and then (\ref{RADIUS}) follows. To prove (\ref{GRAPH}), we first derive the upper bound on the diameter of $N$ which is defined by the {\it intrinsic} distance on $N$ equipped with its induced metric from the Euclidean space. Let $\mathcal{K}$ be the Gauss curvature of $N$. Using the Gauss equation on $N$ in Euclidean space,  
\begin{align*}
		&\left| \mathcal{K} - \frac{1}{2} (H^e)^2 \right| \le   | \mathring{A}^e |^2  \le c H^{2 + 2\epsilon}.
\end{align*}
Hence, $|\mathcal{K} | \ge \frac{1}{8} H^2$, for $H$ small. The Bonnet--Myers theorem says that $\mbox{diam}(N) \le c H^{-1}$. Then, let $z$ be the position vector and $g_N$ be the induced metric on $N$ from Euclidean space. By the Gauss--Weingarten relation
\[
	\partial_i \nu_e = A^e_{ij} g_N^{jk} \partial_k z = \left(\mathring{A}^e_{ij} - \frac{1}{2}H^e(g_N)_{jk} \right) g_N^{jk} \partial_k z.
\]
Then, 
\[
	\partial_i \left(\nu_e - \frac{1}{2}H z\right) =\left[ \mathring{A}^e_{ij} - \frac{1}{2} (H^e - H) (g_N)_{jk}\right] g_N^{jk} \partial_k z.
\]
We integrate the above identity along a geodesic, and derive, for some $p$, 
\[
		|\nu_e - r_0^{-1} (z - p) |\le c\sup_N \left( |\mathring{A}^e| + |H^e - H| \right) \mbox{diam}(N) \le c H^{\epsilon}.
\]

To prove that $N$ is a graph over $S_{r_0}(p)$, we define $v(x) = | z - x |$ where $x \in S_{r_0} (p)$ is the intersection of the ray $z-p$ and $S_{r_0}(p)$. By (\ref{GRAPH}) , for $H$ small, 
\begin{align*}
		\left| \frac{z - p}{r_0} - \nu_e \right| \le \frac{1}{2}. 
\end{align*}
In particular, $\nu_e$ never becomes perpendicular to the radial direction, so $N = \{ x + v \frac{x-p}{r_0}  : x\in S_{r_0}(p)\}$ is well-defined. To obtain the $C^1$ bound on $v$, we have $\big| |z-p| - r_0 \big| \le c H^{-1+\epsilon}$ by (\ref{GRAPH}), and then 
\begin{align*}
		\| v \|_{C^0} =& \sup_{z\in N} | z - x | \le \sup_{z\in N} \big| |z-p| - r_0 \big| \le c H^{-1+\epsilon}.
\end{align*}
Moreover,  
\[
	| \partial v | = | z - x |^{-1} \left| \langle \nabla^e (z-x), z-x\rangle \right| \le | \nabla^e (z-p) - \nabla^e (x-p) |.
\] 
Using \eqref{GRAPH} and that $|\nabla^e (\nu^e - \frac{x-p}{r_0})| \le |\mathring{A}^e|$, 
we obtain
\begin{align*}
		| \partial v | \le  c H^{\epsilon}.
\end{align*}
Therefore, we conclude $\| v^* \|_{C^{1}(S_1(0))} \le c H^{-1+\epsilon}$. Moreover, because $\nu_e$ and $\nu_g$ on $N$ are close in $C^{2,\alpha}$, 
\begin{align*}
		N =  \{ x + \nu_g v : v\in C^1( S_{r_0}(p))\} 
\end{align*}
for some $v$ satisfying $\| v^* \|_{C^{1}(S_1(0)} \le c H^{-1+\epsilon}$.
\end{proof}

In order to use the Taylor theorem to the mean curvature map, $N$ should be a graph whose $C^{2,\alpha}$--norm is under control. Therefore, we have to derive the pointwise estimate on the $C^{1,\alpha}$-norm of $\mathring{A}$. A modified Moser iteration which involves the special choice of the cut-off functions is employed.
\begin{lemma} \label{POINT2}
For any functions $u \ge 0$, $f \ge 0$, and $h$ on $N$ satisfying
\begin{eqnarray} \label{GE2}
				- \Delta u \le f u + h,
\end{eqnarray}
we have the pointwise control on $u$ as follows:
\begin{align*}
		\sup_N u \le c \left(\| f \|_{L^2 }  + H  + \underline{r}^{-1} \right)(  \| u \|_{L^2} + \underline{r}^{2} \| h \|_{L^2 }).
\end{align*}
\end{lemma}
\begin{Remark}
Comparing this lemma with Lemma {\ref{POINT}}, the term $H^{-1} \| h \|_{L^2} = (\underline{r}^{-1})(\underline{r} H^{-1} \| h \|_{L^2})$ is replaced by $\underline{r} \| h \|_{L^2} = (\underline{r}^{-1})(\underline{r}^2 \| h \|_{L^2})$. The term $H^{-1} \| h \|_{L^2}$ is unfavorable because if this term appeared in Corollary \ref{GRADA}, it is bounded by $H^{-1} \underline{r}^{-3-q}$ which may \emph{not} be bounded by $H^{2+\epsilon}$ for $\epsilon >0$, when $2/(2+q) < a \le 1$.
\end{Remark}
\begin{proof}
Let $k$ be a positive constant. As in the proof of Lemma \ref{POINT}, we define $\hat{u} = u+k$ and $\hat{f} = f+k^{-1} h $. Let $\chi$ be a cut-off function on $N$. The same calculations in Lemma {\ref{POINT}} give
\begin{align*}
		\int_N \left| \nabla (\chi \hat{u}^{\frac{p}{2}}) \right|^2 \, d\sigma \le p \int_N \chi^2 \hat{f} \hat{u}^p  \, d\sigma + \int_N | \nabla \chi |^2 \hat{u}^p \, d\sigma.
\end{align*}
By \eqref{SOB} and letting $v = \chi \hat{u}^{p/2}$,  
\begin{align*}
		\left( \int_N \chi^4 \hat{u}^{2p} \, d\sigma \right)^{\frac{1}{2}} \le& c (H + \sup_N | \nabla \chi |) \int_{\mbox{supp}(\chi)} \hat{u}^p \, d\sigma\\
		&+ c \left( \int_N \chi^2\hat{u}^p \, d\sigma \right)^{1/2} \left( p \int \chi^2 \hat{f} \hat{u}^p \, d\sigma \right)^{1/2}.
\end{align*}
Using the AM--GM inequality to the second line and absorbing the  term $\chi^4 \hat{u}^{2p}$ to the left, we obtain
\begin{align*}
		\left( \int_N \chi^4 \hat{u}^{2p} \, d\sigma \right)^{\frac{1}{2}} \le c p \left[ \| \hat{f} \|_{L^2} + H +  \sup_N | \nabla \chi | \right] \int_{\mbox{supp}(\chi)} \hat{u}^p\, d\sigma.
\end{align*}
Let $p_i = 2^i, i =1,2,3,\dots$. Fix $z_0$, the cut-off functions supported on $N$ is defined by, for $z\in N$,
\begin{align*}
		\chi_i (z) = \left\{ \begin{array}{ll} 1 &\quad \mbox{if } z \in B_{(1 + 2^{-i}) \underline{r}  } ( z_0 ) \\
												0 & \quad \mbox{if } z \mbox{ outside } B_{(1 + 2^{-i +1} ) \underline{r} } (z_0)  \end{array} \right.,
\end{align*} 
and $| \nabla \chi_i | \le 2^i \underline{r}^{-1}$. Then
\begin{align*}
		&\left[ \int_{B_{( 1 + 2^{-l} )\underline{r}} (z_0)} \hat{u}^{2^{1+l} } \, d\sigma \right]^{2^{-1-l}} \le c^{\sum_{i=1}^l  2^{-i}}  2^{ \sum_{i=1}^l i2^{-i} } \left[ \| \hat{f} \|_{L^2}^{ \sum_{i=1}^l 2^{-i}}\right.\\
		&\qquad \qquad \qquad \qquad\left.+ H^{ \sum_{i=1}^l 2^{-i}} +\left( 2^{i}\underline{r}^{-1}  \right)^{\sum_{i=1}^l 2^{-i}}  \right] \| \hat{u} \|_{L^2 (B_{2 \underline{r} } (z_0))}.
\end{align*}
Let $l\rightarrow \infty$, 
\[
	\sup_{B_{\underline{r}}(z_0)} \hat{u} \le c (\| \hat{f} \|_{L^2} + H +\underline{r}^{-1} ) \| \hat{u} \|_{L^2(B_{2\underline{r}}(z_0)) }.
\]
Let $k = \underline{r} \| h \|_{L^2}$. Then 
\begin{align*}
		\sup_{B_{\underline{r}} (z_0)} u \le c (\| f \|_{L^2} + H + \underline{r}^{-1})(\| u \|_{L^2} + \underline{r} \| h \|_{L^2} \left| B_{2 \underline{r}}(z_0) \right|^{1/2}).
\end{align*}
 By the area formula, because $g$ is AF, and $N$ is a graph of $v$ over $S_{r_0}(p)$ satisfying $|\partial v | \le c H^{\epsilon}$ by Lemma \ref{POSITION}, 
\begin{align*} 
		|B_{2 \underline{r}} (z_0 )| &=\int_{B_{2 \underline{r}}(z_0)} \, d\sigma  \le \int_{B_{2 \underline{r}}(z_0)} (1+cH^{\epsilon})\, d\sigma_e  \\
		&\le 2\int_{B_{2 \underline{r}}(z_0)}  \, d\sigma_e \le c \underline{r}^{2}.
\end{align*}
\end{proof}

\begin{cor} \label{GRADA}
Assume that $N$ satisfies the assumptions in Theorem \ref{GU1}. Then 
\begin{align*}
		\sup_N | \nabla \mathring{A} | \le  c \underline{r}^{-1-q} (\underline{r}^{-1} + H).
\end{align*}
Moreover, if $\underline{r} \ge  H^{-a}$ for some fixed $a \le 1$, then
\begin{align*}
		\sup_N | \nabla \mathring{A} | \le c H^{2+ \epsilon},
\end{align*}
where $\epsilon = (2+q) a -2 > 0$, if  $\frac{2}{ 2+q} < a \le 1 $ .
\end{cor}
\begin{proof}
Let $T_{\gamma \alpha \beta} = \nabla_{\gamma} \mathring{A}_{\alpha \beta}$. 
\[			
2 | T | \Delta | T | + 2 \left| \nabla | T| \right|^2 = \Delta | T|^2 = 2 T^{\gamma \alpha \beta} \Delta T_{\gamma \alpha \beta} + 2 | \nabla T |^2.
\]
Because $ | \nabla T |^2 \ge \left| \nabla | T| \right|^2$ by the Cauchy--Schwarz inequality, 
\begin{align*}
		 -  | T | \Delta | T | \le - T^{\gamma \alpha \beta} \Delta T_{\gamma \alpha \beta}.
\end{align*}
Changing the order of differentiation in the Laplacian term, 
\begin{align*}
		\Delta ( \nabla_{\gamma} \mathring{A}_{\alpha \beta} )  =& \nabla_{\gamma} \Delta \mathring{A}_{\alpha \beta} + g^{\rho \delta} ( \nabla_{\epsilon} \mathring{A}_{\alpha \beta} ) R^{\;\; \epsilon}_{\delta \; \gamma \rho} + g^{\rho \delta} (\nabla_{\delta} \mathring{A}_{\epsilon \beta}) R^{\;\; \epsilon}_{\alpha \; \gamma \rho} \\
		& + g^{ \rho \delta } (\nabla_{\delta} \mathring{A}_{\alpha \epsilon}) R^{\;\; \epsilon}_{ \beta \; \gamma \rho} + g^{\rho \delta} \nabla_{\rho} \left( \mathring{A}_{\epsilon \beta} R^{\;\; \epsilon}_{\alpha \; \gamma \delta} + \mathring{A}_{\alpha \epsilon} R^{\;\; \epsilon}_{\beta \; \gamma \delta }\right).
\end{align*}
By the Simons identity, 
\begin{align*}
		&\left(\nabla^{\gamma} \mathring{A}^{\alpha \beta } \right) \nabla_{\gamma} \Delta \mathring{A}_{\alpha \beta} \\
		&= H \left(\nabla^{\gamma} \mathring{A}^{\alpha \beta }\right ) \nabla_{\gamma} \left(A^{\delta}_{\alpha} A_{\delta \beta} \right) - \left(\nabla^{\gamma} \mathring{A}^{\alpha \beta } \right) \nabla_{\gamma} \left( |A|^2 A_{\alpha \beta} \right) \\
		&\quad +\left (\nabla^{\gamma} \mathring{A}^{\alpha \beta } \right) \nabla_{\gamma} \left[ A^{\delta}_{\alpha} R^M_{\epsilon \beta \epsilon \delta} + A^{\delta \epsilon} R_{\delta \alpha \beta \epsilon}  + \nabla_{\beta} \left( Ric^M_{\alpha k} \nu^k \right) + \nabla^{\delta} \left( R_{k \alpha \beta \delta } \nu^k  \right) \right].
\end{align*}
Then, 
\begin{align*}
		\left(\nabla^{\gamma} \mathring{A}^{\alpha \beta } \right) \nabla_{\gamma} \Delta \mathring{A}_{\alpha \beta}  &\ge  - | \mathring{A} |^2 | \nabla \mathring{A} |^2\\
		& - c \left(H| \nabla \mathring{A}|^2 |A |  +  | \nabla \mathring{A} |^2 | \mbox{Riem}| + | \nabla \mathring{A} | | \mathring{A} | | \nabla \mbox{Riem} | \right.\\
		&\left.+ | \nabla \mathring{A} | H | | \nabla \mbox{Riem} |  + | \nabla \mathring{A} | | \nabla^2 \mbox{Riem}| + | \nabla \mathring{A}|^2 | \nabla \mbox{Riem} | \right).
\end{align*}
Using $ | \mbox{Riem}|  \le c |x|^{-2-q},  | \nabla \mbox{Riem}|  \le c |x|^{-3-q} $, $  | \nabla^2 \mbox{Riem}|  \le c |x|^{-4-q} $, and combining the above estimates,
\begin{align*}
	 &- |\nabla \mathring{A} | \Delta | \nabla \mathring{A}| \le- \nabla^{\gamma} \mathring{A}^{\alpha \beta} \Delta \left( \nabla_{\gamma} \mathring{A}_{\alpha \beta} \right) \\
		&\le c \left( |\mathring{A}|^2 | \nabla \mathring{A} |^2 + H^2 | \nabla \mathring{A} |^2 + H | \mathring{A} | | \nabla \mathring{A} |^2 +  | \nabla \mathring{A} |^2 |x|^{-2-q} + | \nabla \mathring{A}|^2 |x|^{-3-q}\right.\\
		&\quad \left. + | \nabla \mathring{A} | | \mathring{A} |  |x|^{-3-q}+ H | \nabla \mathring{A} | | x |^{-3-q}+ | \nabla \mathring{A} | |x|^{-4-q}\right).
\end{align*}
By Lemma \ref{POINT2}, we set $ u = | \nabla \mathring{A}|$ and 
\begin{align*}
		& f =  c( | \mathring{A} |^2 + H^2 +H | \mathring{A} | + |x|^{-2-q} + | x |^{-3-q}), \\
		& h = c (| \mathring{A} | |x|^{-3-q} + H |x|^{-3-q} + | x |^{-4-q}).
\end{align*}
By Lemma \ref{BASIC} and Lemma \ref{L2},  $\| u \|_{L^2} \le c \underline{r}^{-1-q}$, $\| f \|_{L^2 }  \le c( \underline{r}^{-1-q} +H)$, and $\| h \|_{L^2} \le c \underline{r}^{-3-q}$. Then the proof follows directly from Lemma \ref{POINT2}.

\end{proof}

Similarly, we can derive that, if $\underline{r} \ge H^{-a}$, then the H\"{o}lder norm $\left[ |\nabla \mathring{A} | \right]_{\alpha} \le c H^{2 +\alpha + \epsilon}$, where $\epsilon = (2+q)a-2$ and $\epsilon >0 $ if $2/(2+q) < a \le 1$. Using the same argument in Lemma \ref{POSITION}, we prove the following:
\begin{cor}\label{C2GRAPH}
Assume that $N$ satisfies the assumptions in Theorem \ref{GU1}. If $\underline{r} \ge c H^{-a}$ for some $2/(2+q) < a\le 1$. Then $N$ is a graph defined by $N = \{ x + \nu_g v : x\in S_{r_0}(p)\} $ with 
\[
	\|v^* \|_{C^{2,\alpha}(S_1(0))} \le c H^{-1+\epsilon} \le c r_0^{1-\epsilon},
\]
where $\epsilon = (2+q)a -2 > 0$.
\end{cor}


\subsection{Global Uniqueness}

\begin{proof}[Proof of Theorem \ref{GU1} ]
Assume that $N$ has mean curvature equal to $H = H_{\Sigma_R}$ for some $R$. By Corollary \ref{LU2}, we only need to prove that $N$ is a graph of $v$ over $S_R(p)$ where $|p - \mathcal{C}| \le c_3 R^{1-q}$ and $\|v^*\|_{C^{2,\alpha}} \le c_4R^{1-q}$ for some $c_3$ and $c_4$. 

By Corollary \ref{C2GRAPH}, $N$ is a graph over $S_{r_0}(p)$ and $H = 2/r_0$. The idea of the proof is to show that the center $p$ does \emph{not} drift away too much as $H$ goes to zero. More precisely, we show that $|p-\mathcal{C}| \le c r_0^{1-\epsilon'}$ for some $\epsilon' >0$. Hence, $r_0$ and $\underline{r}$ are comparable; consequently, $H$ and $\underline{r}$ are comparable. 

By the Taylor theorem, because $N$ is the graph of $v$ over $S_{r_0}(p)$, 
\begin{align*}
		&H = H _{S } - L_S  v+ \int_0^1\left( d H (s v) - d H(0) \right)v \, ds.
\end{align*}
Recall that $L_S = - \Delta_S - (| A_S|^2 + Ric^M(\nu_g, \nu_g))$, $L_0 = - \Delta_S^e - \frac{2}{R^2}$, and $\mathfrak{K} = \mbox{Ker}L_0$. Also recall that $\phi$ in Lemma \ref{lemma:approx}, and $\phi$ satisfies
\[
	L_0 \phi = f - r_0^{-3-q} \sum_i A^i (x^i - p^i) - \overline{f}, 
\]
where $f =  H _{S}  - 2/r_0$. Therefore, 
\begin{align} \label{eq:uminusphi}
	L_0 (v - \phi) =& r_0^{-3-q} \sum_i A^i (x^i - p^i) + \overline{f} + (L_0 - L_S) v \notag \\
	&+  \int_0^1\left( d H (s v) - d H(0) \right)v \, ds.
\end{align}
We decompose $v = v^{\perp} + r_0^{-\epsilon} v_{\mathfrak{K}}$, where $v_{\mathfrak{K}} = \sum_i B^i (x^i - p^i) \in \mathfrak{K}$ and
\begin{align*}
		B^i = \frac{3 r_0^{-4 + \epsilon }} {4 \pi } \int_{S_{r_0}(p)}  (x^i - p^i) v\, d\sigma_e = O(1).
\end{align*}
Because $v^{\perp} - \phi \in \mathfrak{K}^{\perp}$, we apply the Schauder estimate to \eqref{eq:uminusphi}, 
\begin{align*}
		\| v^{\perp} - \phi \|_{C^{2,\alpha}} \le  c ( r_0^{1- 2\epsilon} + r_0^{1-q}).
\end{align*}
Without loss of generality, we assume $\epsilon < q/2$. The right hand side of the above identity is dominated by $c r_0^{1-2\epsilon}$. Consider $(v^{\perp} -\phi)^{odd}$:
\begin{align} \label{eq:odd}
	&	L_0( (v^{\perp} - \phi)^{odd}) = 2 r_0^{-3-q} \sum_i A^i (x^i - p^i)  + (L_0 - L_S) (v^{\perp})^{odd} \notag \\
	&- 2 L_S  (r_0^{-\epsilon} u_{\mathfrak{K}})  +  (L_S - L_0)^{odd} v\notag \\
	&+  \left[\int_0^1\left( d H (s v) - d H(0) \right) \, ds\right]^{odd} v + \int_0^1\left( d H (s v) - d H(0) \right) v^{odd} \, ds.
\end{align}

 Then by the Schauder estimate, and using Lemma \ref{lemma2} and Lemma \ref{lemma3} to estimate the last three terms,  
\begin{align*}
		&\| ((v^{\perp})^*)^{odd} \|_{C^{2,\alpha}} \le \| (\phi^*)^{odd} \|_{C^{2,\alpha}} + c\left(r_0^{-q} +  r_0^{-q} \| (v^*)^{odd} \|_{C^{2,\alpha}}+ r_0^{1-q-\epsilon}\right.\\
		&\qquad \qquad \left.  + r_0^{-1-q} \| v^* \|_{C^{2,\alpha}}  + r_0^{-1} \|((v^{\perp})^*)^{odd} \|_{C^{2,\alpha}} \| v \|_{C^{2,\alpha}} \right).
\end{align*}
Bootstrapping the term  $\| ((v^{\perp})^*)^{odd} \|_{C^{2,\alpha}} $ yields
\[
	\| ((v^{\perp})^*)^{odd} \|_{C^{2,\alpha} (S_1(0))} \le c \underline{r}^{1-q-\epsilon}.
\]
Integrating the both sides of \eqref{eq:odd} with $x^a - p^a$ on  $S_{r_0}(p)$ with respect to the area measure $d\sigma$. Because $d \sigma = (1+O(r_0^{-q})) d \sigma_e$, 
\[
	\int_{S_{r_0}(p)} (x^a - p^a) L_0 ((v^{\perp} - \phi)^{odd}) \, d\sigma = O(r_0^{2-2q-\epsilon}).
\]
By the definition of $A^i$ and Lemma \ref{lemma1},
\[
	\int_{S_{r_0}(p)} (x^a - p^a)  r_0^{-3-q} \sum_i A^i (x^i - p^i)  \, d\sigma = 8\pi m (p^a - \mathcal{C}^a) + O(r_0^{1-2q}).
\]
Also, by Lemma \ref{STABLE} (there, the equality that $\mu_0 = 6\pi m/r_0^3 + O(r_0^{-2-2q})$ is achieved by the coordinate functions $x^i - p^i$), 
\[
	 r_0^{-\epsilon} \int_{S_{r_0}(p)} (x^a - p^a) L_S \sum_i B^i (x^i - p^i)  \, d\sigma =r_0^{-\epsilon} \frac{6\pi m}{ r_0^3} B^a r_0^4 + O(r_0^{2-2q-\epsilon}).
\]
The rest terms are of order $O(r_0^{2-2q-\epsilon})$. 
Therefore, we have
\[
	|p^a - \mathcal{C}^a | \le c \left(r_0^{1-\epsilon} + r_0^{1 - (q+2\epsilon-1)}\right).
\]
Recall that $\epsilon = (2+q) a -2$. By the assumption that $a > (5-q)/2(2+q)$, we have $\epsilon':=q+2\epsilon-1>0$.
Then $| p^a| \le c r_0^{1-\epsilon'}$, so the center $p$ may drift away but at a \emph{controlled} rate. Let $z_0$ be a point so that $\underline{r} = |z_0|$,
\begin{align*}
		\underline{r} = | z_0 | \ge | z_0 - p | - | p | \ge r_0 -c H^{-1+\epsilon}- c r_0^{1-\epsilon'}.
\end{align*}
For $r_0$ large, $\underline{r} \ge c r_0$. Therefore, we can replace the assumption $\underline{r} \ge H^{-a} $ by $\underline{r} \ge c r_0 \ge cH^{-1}$ in Corollary \ref{C2GRAPH}. 
Therefore, $N$ is a $c r_0^{1-q}$-graph over $S_{r_0}(p)$ and $| p - \mathcal{C} | \le cr_0^{1-q}$. Although $H$ may not be exactly equal to $H_{\Sigma_{r_0}}$, we can choose $R$ so that $H = H_{\Sigma_{R}}$ with $R = r_0 + O(r_0^{-q})$. Then we can apply the local uniqueness result of Corollary \ref{LU2} by viewing $N$ as a graph over $S_{R}(p)$ and conclude $N = \Sigma_{R}$.
\end{proof}

To prove a result of the uniqueness outside a {\it fixed} compact set, we replace the condition on $\underline{r}$ by the condition that $\overline{r}$ and $\underline{r}$ satisfy $\overline{r} \le c_2 \underline{r}^{a^{-1}}$ for any $(5-q) / 2(2 + q) < a \le 1$.
\begin{proof}[Proof of Theorem \ref{GU2}]
If $N$ lies completely outside $B_{ H^{-a}}(0)$ for some $a$ satisfying $(5-q) / 2(2 + q)< a \le 1$, by Theorem \ref{GU1}, $N = \Sigma_R$. We assume that $ N \neq \Sigma_R$. Therefore $N \cap B_{H^{-a} }(0) \neq \phi$ for any $(5-q) / 2(2 + q)< a \le 1$. Then $\underline{r} \le H^{-a} \le 3 R^a$ if $R$ large enough because $H = (2/R ) + O(R^{-1-q})$. On the other hand, for any $z\in N$,
\begin{align*}
		\frac{2}{ \overline{r} } \le H^e(z) \le 2H  \le \frac{4}{R} + c R^{-1-q} .
\end{align*}
For $R$ large, 
\begin{align*}
		\frac{2}{ \overline{r} } \le \frac{6}{R},
\end{align*}
and then $R/3 \le \overline{r}$. Therefore, 
\begin{align*}
		\frac{1}{(3)^{\frac{1}{a}-1}} (\underline{r} )^{\frac{1}{a} } \le	\frac{1}{(3)^{\frac{1}{a} -1 }} (3 R^a)^{\frac{1}{a}} \le \frac{R}{3} \le \overline{r}.
\end{align*}
Choosing any $c_2 <\frac{1}{\sqrt{3}}$, we obtain $c_2 \underline{r}^{\frac{1}{a}} < \overline{r}$ which contradicts to the assumption. Therefore, $N = \Sigma_R$.
\end{proof}

\section*{Acknowledgments}
I would like to thank my advisor Professor Rick Schoen for suggesting this problem and providing useful comments. I also would like to thank Professors Brian White, Leon Simon, Damin Wu, Justin Corvino, Jan Metzger, Gerhard Huisken for discussions. Also, I thank the referee for suggestions and for pointing out several typographical errors.

\bibliographystyle{plain}
\bibliography{20100331mybib}

\end{document}